\documentclass[12pt]{amsart}
\usepackage{amscd,amsmath,amsthm,amssymb,mathtools}
\usepackage{mathrsfs}
\usepackage{comment}
\usepackage[all,cmtip]{xy}
\usepackage{color}
\usepackage{enumitem}

\definecolor{cadmiumgreen}{rgb}{0.0, 0.42, 0.24}
\usepackage[
colorlinks, citecolor=cadmiumgreen,
pagebackref,
pdfauthor={Farbod Shokrieh, Chenxi Wu}, 
pdftitle={Canonical measures on metric graphs and a Kazhdan's theorem},
pdfstartview ={FitV},
]{hyperref}

\usepackage[
alphabetic,
backrefs,
msc-links,
nobysame,
lite,
]{amsrefs} 

\usepackage[left=2.5cm,top=2.5cm,right=2.5cm]{geometry}

\usepackage{tikz, float} \usetikzlibrary {positioning}
\usetikzlibrary{calc,decorations.markings}
\usetikzlibrary{shapes,snakes}

\def\ZZ{{\mathbb Z}}
\def\RR{{\mathbb R}}
\def\CC{{\mathbb C}}
\def\EE{{\mathbb E}}

\def\B{{\mathcal B}}
\def\Oc{{\mathcal O}}

\def\A{{\mathcal A}}

\def\M{{\mathcal M}}

\def\G{{\mathbf G}}

\def\Rsc{{\mathscr R}}

\def\opn#1#2{\def#1{\operatorname{#2}}} 
\opn\depth{depth} 
\opn\codim{codim}
\opn\ini{in} 
\opn\LM{LM}
\opn\LC{LC}
\opn\NF{NF}
\opn\Merge{Merge}
\opn\sgn{sgn}
\opn\div{div} 
\opn\Div{Div} 
\opn\Pic{Pic}
\opn\Prin{Prin}
\opn\Del{Del}
\opn\op{op}
\opn\ends{ends}
\opn\indeg{indeg} 
\opn\outdeg{outdeg}
\opn\red{red}
\opn\Spec{Spec} 
\opn\Supp{Supp} 
\opn\Meas{Meas}
\opn\supp{supp} 
\opn\Ker{Ker} 
\opn\Coker{Coker} 
\opn\sign{sign}
\opn\Hom{Hom}
\opn\Tor{Tor} 
\opn\id{id}
\opn\cl{cl}
\opn\span{span}
\opn\trace{trace}
\opn\Image{Image}
\opn\con{conv} 
\opn\relint{rel.int} 
\opn\vol{vol}
\opn\val{val}
\opn\Ber{Ber}
\opn\can{can}
\opn\syz{{\rm syz}}
\opn\spoly{{\rm spoly}}
\opn\LM{{\rm LM}}
\opn\lm{{\rm lm}}
\opn\lcm{{\rm lcm}} 
\opn\A{\mathcal A}
\opn\L{\mathrm{L}}
\opn\dist{dist}
\opn\pd{pd}
\opn\en{en}
\opn\PL{\mathrm{PL}}
\opn\H{\mathrm{H}}
\opn\V{\mathrm{V}}
\opn\B{\mathcal{B}}
\opn\Tr{\mathrm{Tr}}
\opn\Gr{\mathbf{Gr}}

\def\Implies{\ifmmode\Longrightarrow \else
        \unskip${}\Longrightarrow{}$\ignorespaces\fi}
\def\implies{\ifmmode\Rightarrow \else
        \unskip${}\Rightarrow{}$\ignorespaces\fi}
\def\iff{\ifmmode\Longleftrightarrow \else
        \unskip${}\Longleftrightarrow{}$\ignorespaces\fi}
\newtheorem{Theorem}{Theorem}[section]
\newtheorem{Lemma}[Theorem]{Lemma}
\newtheorem{Corollary}[Theorem]{Corollary}
\newtheorem{Proposition}[Theorem]{Proposition}

\theoremstyle{remark}
\newtheorem{Remark}[Theorem]{Remark}

\theoremstyle{definition}
\newtheorem{Example}[Theorem]{Example}

\newtheorem{Definition}[Theorem]{Definition}

\def\qed{\ifhmode\textqed\fi
      \ifmmode\ifinner\quad\qedsymbol\else\dispqed\fi\fi}
\def\textqed{\unskip\nobreak\penalty50
       \hskip2em\hbox{}\nobreak\hfil\qedsymbol
       \parfillskip=0pt \finalhyphendemerits=0}
\def\dispqed{\rlap{\qquad\qedsymbol}}

\tikzstyle{Cwhite}=[scale = .8,circle, fill = white, minimum size=3mm] 
\tikzstyle{Cgray}=[scale = .4,circle, fill = gray, minimum size=3mm] 
\tikzstyle{Cblack2}=[scale = .4,circle, fill = black, minimum size=5mm] 
\tikzstyle{Cblack}=[scale = .7,circle, fill = black, minimum size=3mm]
\tikzstyle{C0}=[scale = .9,circle, fill = black!0, inner sep = 0pt, minimum size=3mm]
\tikzstyle{C1}=[scale = .7,circle, fill = black!0, inner sep = 0pt, minimum size=3mm]
\tikzstyle{Cred}=[scale = .4,circle, fill = red, minimum size=3mm]

\begin{document}

\title[Canonical measures on metric graphs and a Kazhdan's theorem]{Canonical measures on metric graphs and \\ a Kazhdan's theorem}

\author {Farbod Shokrieh}
\address{Cornell University\\
Ithaca, New York 14853-4201\\
USA}
\email{\href{mailto:farbod@math.cornell.edu}{farbod@math.cornell.edu}}

\author{Chenxi Wu}
\address{Rutgers University\\
Piscataway, New Jersey 08854-8019\\USA}
\email{\href{mailto:wuchenxi2013@gmail.com}{wuchenxi2013@gmail.com}}


\date{\today}

\begin{abstract}
We extend the notion of canonical measures to all (possibly non-compact) metric graphs. This will allow us to introduce a notion of ``hyperbolic measures'' on universal covers of metric graphs. Kazhdan's theorem for Riemann surfaces describes the limiting behavior of canonical (Arakelov) measures on finite covers in relation to the hyperbolic measure. We will prove a generalized version of this theorem for metric graphs, allowing any infinite Galois cover to replace the universal cover. We will show all such limiting measures satisfy a version of Gauss-Bonnet formula which, using the theory of von Neumann dimensions, can be interpreted as a ``trace formula''. In the special case where the infinite cover is the universal cover, we will provide explicit methods to compute the corresponding limiting (hyperbolic) measure. Our ideas are motivated by non-Archimedean analytic and tropical geometry.
\end{abstract}

\maketitle

\setcounter{tocdepth}{1}
\tableofcontents

\section{Introduction}

\subsection{Background}
Metric graphs, in many respects, can be thought of as non-Archimedean analogues of Riemann surfaces. The analogy between metric graphs and Riemann surfaces can be made precise, for example, by using Berkovich's theory of non-Archimedean analytic geometry or by using tropical geometry; any metric graph may be thought of as a skeleton for a Berkovich analytic curve.
For metric graphs, there is a well-behaved theory of divisors and Jacobians (\cite{BN07, MK, ABKS}). In geometric group theory, the analogies between the Teichm\"uller space and the outer space (\cite{culler1986moduli}), the curve complex and the free factor (free splitting, cyclic splitting) complex (\cite{handel2013free, bestvina2014hyperbolicity, mann2014hyperbolicity}), and Teichm\"uller polynomials and McMullen polynomials (\cite{dowdall2015dynamics}) are subjects of active research.

In this paper, we pursue this analogy further in the context of metrics and measures.
For example, a basic question that arises from this analogy is the following:
what is the correct notion of the hyperbolic measure on metric graphs of genus at least $2$? Our main results, in particular, will provide a satisfactory answer to this question. Our approach is guided by non-Archimedean analytic and tropical considerations.

It is instructive to first consider the case of Riemann surfaces. 
A compact Riemann surface $X$ of genus at least $2$ can be given canonical metrics in a few different ways. One obvious choice is by uniformization; since the Poincar\'e metric on the unit disc (the universal cover) is invariant under the action of ${\rm PSL}(2, \mathbb{R})$, it induces a metric on $X$, which is the usual {\em hyperbolic metric}. We will refer to the associated volume form as the {\em hyperbolic measure}. 
Another choice is by embedding the surface inside its Jacobian, and pulling back the Euclidean metric. This metric is referred to as the {\em canonical metric}, the {\em Bergman metric}, or the {\em Arakelov metric} in the literature. The associated volume form is usually called the {\em canonical measure}, or the {\em Arakelov measure}.
A natural question is the relationship between these two metrics (resp. measures). A celebrated theorem of Kazhdan (\cite[\S3]{Kaj}, see also \cite{mumford, Kazhdan, rhodes1993sequences, mcmullen2013entropy}) states that the hyperbolic metric (resp. the hyperbolic measure) is, up to scaling, the limit of the canonical metrics (resp. the canonical measures). More precisely, let $X_n \rightarrow X$ be a cofinal ascending sequence of finite Galois covers which converge to the universal cover. Then the metrics on $X$ inherited from canonical metrics under $X_n \rightarrow X$ converge to a multiple of the hyperbolic metric. It is well-known that the canonical metric on the unit disc is the same as the hyperbolic metric (up to a constant multiple). So this result can be restated as follows: the limit of the induced canonical metrics (resp. canonical measures) coincides with the induced canonical metric (resp. canonical measures) from the limiting space (the universal cover). So, informally, Kazhdan's theorem can be interpreted as saying that the two tasks ``taking limit'' and ``computing the inherited metrics (or measures)'' commute.

For metric graphs, we will adopt the language of measures instead of metrics. There is already a satisfactory analogue of the notion of canonical measures for {\em compact} metric graphs. This was introduced by Zhang in \cite{zhang1993admissible} (see also \cite{chinburg1993capacity}) in terms of potential theory (electrical networks) on graphs (see Definition~\ref{def:zhang}). The notion was related to the Abel-Jacobi map in tropical geometry by Baker and Faber in \cite{baker2011metric}. The analogy goes far; canonical measures on metric graphs are closely related to the non-Archimedean Monge-Amp\`ere (Chambert-Loir) measures on Berkovich curves (see \cite[3.2.3]{ChLoir} and \cite{Heinz}). The Archimedean version of this statement is also true: the canonical measures are closely related to the Archimedean Monge-Amp\`ere measures on Riemann surfaces. Furthermore, Amini in \cite{amini2014equidistribution} proves that the limiting distribution of Weierstrass points on a Berkovich curve is closely related to the canonical measure on its skeleton (which is a compact metric graph). This result corresponds to a similar result on Riemann surfaces due to Mumford and Neeman (\cite{mumford}, \cite{Neeman}). In fact, one expects that the canonical measure on metric graphs should be a ``limit'' of canonical measures on Riemann surfaces, where the limit should perhaps be taken in the sense of Boucksom and Jonsson (\cite{BJ}). See \cite{deJong} (especially Remark~16.4) for a result in this direction.

\subsection{Our contribution}
Our first contribution is to extend the notion of canonical measures to {\em all} metric graphs. The existing definitions in terms of electrical networks or in terms of tropical Jacobians does not work for non-compact metric graphs. 
We will give two new definitions (in Definition~\ref{def:2} and in Proposition~\ref{prop:mcmul}), which make sense for all metric graphs, including non-compact ones. 

As a first application of this notion, one obtains an appropriate candidate for the notion of {\em hyperbolic measures} for metric graphs.
As already mentioned, there is no satisfactory analogue of this notion for metric graphs, although candidates can certainly be found in the literature (see e.g. \cite[Appendix]{McMullenEntr} and \cite{Feng}). 
Motivated by the above discussion on Riemann surfaces and our non-Archimedean considerations, we will {\em define} the hyperbolic measure to be the measure induced from the canonical measure on the universal cover.

Having established these notions, one might hope to have an analogous theorem to Kazhdan's for metric graphs. We will, in fact, prove a much stronger statement for metric graphs: 

\vspace{3mm}
\noindent{\bf Theorem A.} 
Let $\Gamma'$ be {\em any infinite Galois cover} of a compact metric graph $\Gamma$. Let $\{\Gamma_n\}$ be a sequence of finite Galois covers converging to $\Gamma'$. Then the following two measures on $\Gamma$ coincide:

(1) limit of measures induced from the canonical measures on $\Gamma_n$, and 

(2) the measure induced from the canonical measure on $\Gamma'$.

\vspace{1mm}
See Theorem~\ref{thm:Kazhdan} for a precise statement. 
\vspace{2mm}

It is easy to check on examples (see e.g. Example~\ref{ex:amusing}) that our measure is different from the candidates for hyperbolic measures mentioned above. Most importantly, other measures will not satisfy the desired analogue of Kazhdan's theorem.

A fundamental property of the hyperbolic measure on compact Riemann surfaces is the well-known consequence of the Gauss-Bonnet formula that the total mass of the measure has a simple expression in terms of the Euler characteristic of the surface. Our second main theorem states that {\em all} limiting measures coming from infinite Galois covers satisfy the following analogue of the Gauss-Bonnet formula:

\vspace{3mm}
\noindent{\bf Theorem B.} 
Let $\Gamma$ be a compact metric graph, and $\Gamma'$ be an arbitrary infinite Galois cover. Let $\mu$ denotes the measure on $\Gamma$ induced from the canonical measure on $\Gamma'$. Then 
$\mu(\Gamma) = -\chi(\Gamma)$,
where $\chi(\Gamma)$ denotes the Euler characteristic of the metric graph $\Gamma$.

\vspace{1mm}
See Theorem~\ref{thm:GaussBonnet} for a precise statement.
\vspace{2mm}

Indeed, Theorem B is also the most technical ingredient in the proof of Theorem A.
This is a rather technical result if one restricts oneself to elementary tools: if the infinite cover is {\em transient} (in the sense of Definition~\ref{def:transient}), the theorem could be proved by a difficult computation. If the infinite graph is not transient, the result is significantly more complicated.
In this paper, we manage to avoid all the difficulties by taking a more sophisticated approach, namely the theory of von Neumann algebras and dimensions (see \S\ref{sec:gdim}). This also has the advantage of interpreting our Gauss-Bonnet identity as a {\em trace formula}. Moreover, it follows that one can interpret the Gauss-Bonnet identity as giving a canonical {\em measure-theoretic decomposition} of the first $L^2$ Betti number of $\Gamma$ (see Remark~\ref{rmk:GB}).

While the universal cover of all Riemann surfaces of genus at least $2$ is the Poincar\'e disc, different metric graphs have different universal covers. This means, for universal covers, our limiting (hyperbolic) measure does closely depend on the geometry of the particular universal cover. We will therefore study this case in some detail. For example, we prove the following result.

\vspace{3mm}
\noindent{\bf Theorem C.} 
On a compact metric graph $\Gamma$, the hyperbolic measure can be computed by solving an explicit collection of algebraic equations described in terms of the combinatorics of $\Gamma$.

\vspace{1mm}
See Theorem~\ref{cal1} and Theorem~\ref{cal2} for precise statements.
\vspace{2mm}

We also show the following result about {\em regular graphs} (Corollary~\ref{cor:reg}). 

\vspace{3mm}
\noindent{\bf Theorem D.} 
Let $k\geq 3$, and assume $G$ is a finite weighted graph with uniform edge lengths which is $k$-regular. Then the hyperbolic measure $\mu$ on the corresponding metric graph satisfies $\mu(e)={1-2/k}$ for all edges $e$ in $G$.

\vspace{1mm}
The fact that the hyperbolic measures for regular graphs should give equal mass to edges was anticipated by Curtis McMullen in a private communication. 
\vspace{2mm}

Finally, we will show that our hyperbolic measure is closely related to other important measures on the universal cover:

\vspace{3mm}
\noindent{\bf Theorem E.} 
On a compact metric graph $\Gamma$, the hyperbolic measure can be explicitly computed in terms of a {\em Poisson-Jensen} (or {\em equilibrium}) measure on its universal cover.

\vspace{1mm}
See Theorem~\ref{PJ} for a precise statement.
\vspace{2mm}

\subsection{Further directions}

From the non-Archimedean point of view, the universal cover of a graph is, in fact, the analogue of the Schottky cover of a Riemann surface. More precisely, if a compact metric graph $\Gamma$ is considered as a skeleton of a Berkovich analytic Mumford curve $X^{\rm an}$, then the universal cover $\Gamma'$ is the geodesic hull of the limit points of the action of the Schottky group on the Berkovich analytic $\mathbf{P}^{1, {\rm an}}$ (see \cite{Berk} and \cite{GvdP}). 

This leads one to wonder if one could also prove a Kazhdan type theorem for the Schottky cover of Riemann surfaces. In an upcoming paper (\cite{BSW}), the authors (together with Hyungryul Baik) give a positive answer to this question. In fact, we show that for {\em any infinite Galois cover} of a Riemann surface, one has a Kazhdan theorem analogous to Theorem A. The proof uses some of the fundamental ideas in this paper, adapted to the Archimedean situation. For example, we also prove an analogue of Theorem B for Riemann surfaces, using ideas from von Neumann algebras. 

McMullen, in \cite[Appendix A]{McMullenEntr}, makes an explicit connection between the notion of optimal metrics on graphs and the thermodynamics formalism. It is conceivable that our discussion on equilibrium measures in \S\ref{sec:PJEquil} may fit into that perspective as well.

Our results generalize easily to ``augmented metric graphs'' (in the sense of \cite{ABBR1}), and  ``admissible measures'' (in the sense of \cite{zhang1993admissible}). See Remark~\ref{rem:admissiblemeas}.

Finally, one might expect our techniques could be slightly modified to prove a Kazhdan's theorem for other types of limits of graphs, most importantly Benjamini-Schramm limits (in the sense of \cite{BenSch, ABBGNRS, ACFK}). On another direction, motivated by the viewpoint offered by Berkovich's theory, one might want to prove similar theorems for unramified harmonic covers (in the sense of \cite{ABBR1, ABBR2}) which include topological covers. One might also want to consider covers corresponding to ``tempered fundamental groups'' (in the sense of \cite{Andre,Lepage}). At the moment, a clear description of the notion of the ``universal cover'' (or any limiting cover) in these generalized situations seems to not have been studied.

\subsection{Structure of the paper}
In \S\ref{sec:Background} we will review some basic notions and properties of metric graphs. Our main focus is to appropriately extend the existing notions for compact metric graphs to all (possibly non-compact) metric graphs.
In \S\ref{sec:forms} the notions of (piecewise linear) forms, functions, and measures on metric graphs are studied. Again, since non-compact metric graphs will appear in our work, care must be taken to appropriately generalize these tropical notions.
In \S\ref{sec:gdim} we quickly review the parts of the theory of von Neumann algebras and dimensions that arise in the presence of group actions, which is the situation appearing in our context.
In \S\ref{sec:canon} we first review the notion of canonical measures on compact metric graphs, including the (original) formula in terms of electrical networks due to Zhang. We then give our generalization to all metric graphs. Next, some basic properties of the canonical measure are established, including the behavior under contraction maps, and an approximation theorem in terms of exhaustions by compact subgraphs. The most important result in this section is the proof of our Gauss-Bonnet type formula (Theorem B).
In \S\ref{sec:kazh} our generalized Kazhdan's theorem for metric graphs (Theorem A) is proved.
In \S\ref{sec:examples} we provide some basic examples, including the case of regular graphs (Theorem D). We then study the case of universal covers closely, describing how one can compute the corresponding induced (hyperbolic) measures (Theorem C). We also make precise connections with certain Poisson-Jensen or equilibrium measures (Theorem E).

\subsection*{Acknowledgments}  We would like to thank Matthew Baker and Curtis McMullen for helpful remarks and suggestions. We also thank Omid Amini, Matthew Baker, Robin de Jong, and the anonymous referee for valuable comments on an earlier draft.

\section{Metric graphs, models, and Galois covers}
\label{sec:Background}
\subsection{Metric graphs}
\begin{Definition}
A {\em metric graph} (or a {\em locally finite abstract tropical curve}) $\Gamma$ is a connected metric space such that every point $p\in\Gamma$ has a neighborhood isometric to a star-shaped set of finite valence and of radius at least $\epsilon$, for some fixed $\epsilon >0$, endowed with the path metric.
\end{Definition}

By a {\em star-shaped} set of finite {\em valence} $n$ and {\em radius} $r$ we mean a set of the form
\[
S(n, r) = \{ z \in \CC \colon z=te^{k \frac{2\pi i}{n}} \text{for some } 0 \leq t < r \text{ and some } k \in \ZZ \} \, ,
\]
where distance between two points in $S(n,r)$ is the length of the shortest path connecting the two points.

In tropical geometry, metric graphs are usually compact. We do not require them to be compact in our definition because we will also be working with infinite covers of (compact) metric graphs.

\subsection{Models}

\begin{Definition} \label{def:metricgraph}
A \emph{weighted graph} $G$ is a connected combinatorial graph (i.e. 1-dimensional, locally finite simplicial complex) with a positive $\RR$-valued length function $\ell$ defined on its edge set. 
\end{Definition}

As usual, the vertex set and the edge set of $G$ will be denoted by $V(G)$ and $E(G)$. A {\em bridge} is an edge $e \in E(G)$ such that $G\backslash e$ is disconnected.

Any weighted graph $G$ gives rise to a metric space $\Gamma$ as follows: for each edge $e \in E(G)$, make an open line segment of length $\ell(e)$, then identify the different ends of different line segments according to the combinatorics of $G$. If we further assume that the length function $\ell$ is uniformly bounded below, then the metric space $\Gamma$ obtained in this way is a metric graph in the sense of Definition~\ref{def:metricgraph}.

Under this construction, a subgraph of $G$ will give a closed subset of the corresponding metric space, which we will also refer to as a {\em subgraph} of $\Gamma$.

Let $\EE(G)$ denote the set of oriented edges of $G$; for each edge in $E(G)$ there are two edges $e$ and $\bar{e}$ in $\EE(G)$. An element $e$ of $\EE(G)$ is called an {\em oriented} edge, and $\bar{e}$ is called the {\em opposite} of $e$. For an oriented edge $e \in \EE(G)$ we will use the same letter $e \in E(G)$ to denote the underlying (unoriented) edge. 
We have a map 
\[
\begin{aligned}
\EE(G) &\rightarrow V(G) \times V(G) \\
e &\mapsto (e^{-}, e^{+})
\end{aligned}
\]
sending an oriented edge $e$ to its tail (or its initial vertex) $e^{-}$ and its head (or its terminal vertex) $e^{+}$. An {\em orientation} of $G$ is a choice of subset $\Oc \subset \EE(G)$ such that $\EE(G)$ is the disjoint union of $\Oc$ and $\bar{\Oc}=\{\bar{e} \colon  e \in \Oc \}$.

A {\em walk} $\beta$ in $G$ is an alternating sequence of vertices $v_i$ and oriented edges $e_i$, 
\[v_0 , e_0 , v_1 , e_1 , v_2 , \ldots , v_{k-1} , e_{k-1} , v_k\] such that $(e_i)^- = v_i$ and $(e_i)^+ = v_{i+1}$. A {\em closed} walk is one that starts and ends at the same vertex.

\begin{Definition}
Given a metric graph $\Gamma$, by a {\em vertex set} we mean a discrete  $V\subset\Gamma$ which includes all the points of valence not equal to $2$, such that all the connected component of $\Gamma \backslash V$ are isometric to bounded open intervals. 
\end{Definition}

With the choice of vertex set $V \subset\Gamma$, one can build a weighted graph $G$ as follows. Let $V(G)=V$. Identify $E(G)$ with the connected components of $\Gamma\backslash V$. Note that, by assumption, all these connected components are isometric to open intervals. The length function on $E(G)$ will record the distance between the endpoints of the corresponding interval in $\Gamma$. An edge $e \in E(G)$ corresponding to an open interval $I \subseteq \Gamma$, connect the point $u$ and $v$ (not necessarily distinct) in $\Gamma$ if and only if $\{ u , v\} = \cl(I) \backslash I$. Here $\cl(\cdot)$ denotes the topological closure. As a result, the metric space obtained from $G$ is isometric to $\Gamma$. See Figure~\ref{fig:models.ex}.

\begin{Definition}
\begin{itemize}
\item[]
\item[(i)] Given a metric graph $\Gamma$, we call the weighted graph $G$ constructed as above a {\em model} for $\Gamma$. 
\item[(ii)] Let $G$ and $G'$ be two models for $\Gamma$. We call $G'$ a {\em refinement} of $G$ if $V(G) \subseteq V(G')$.
\end{itemize}
\end{Definition}

Models for $\Gamma$ form a directed set under refinement, as any two models for $\Gamma$ have a common refinement. 

\begin{figure}[h!]
$$
\begin{xy}
(0,0)*+{
	\scalebox{.7}{$
	\begin{tikzpicture}
	\draw[black, ultra thick, -] (0,1.2) to [out=-45,in=90] (.6,-.05);
	\draw[black, ultra thick] (.6,0) to [out=-90,in=45] (0,-1.2);
	\draw[black, ultra thick, -] (0,1.2) to [out=-135,in=90] (-.6,-.05);
	\draw[black, ultra thick] (-.6,0) to [out=-90,in=135] (0,-1.2);
	\draw[black, ultra thick, -] (0,1.2) -- (0,0);
	\draw[black, ultra thick] (0,0.1) -- (0,-1.2);
	\end{tikzpicture}
	$}
};
(-6.5,0)*+{\mbox{{\smaller $2$}}};
(2,-1)*+{\mbox{{\smaller $1$}}};
(6.75,-0.05)*+{\mbox{{\smaller $2$}}};
(0,-15)*+{\Gamma};
\end{xy}
\ \ \ \ \ \ \ \ \ \ 
\begin{xy}
(0,0)*+{
	\scalebox{.7}{$
	\begin{tikzpicture}
	\draw[black, ultra thick, ->] (-1.2,0) to (-.6,.6);
	\draw[black, ultra thick] (-.6,.6) to (0,1.2);
	\draw[black, ultra thick, ->] (-1.2,0) to (0,0);
	\draw[black, ultra thick] (0,0) to (1.2,0);
	\draw[black, ultra thick, ->] (-1.2,0) to (-.6,-.6);
	\draw[black, ultra thick] (-.6,-.6) to (0,-1.2);
	\draw[black, ultra thick, ->] (0,1.2) to (.6,.6);
	\draw[black, ultra thick] (.6,.6) to (1.2,0);
	\draw[black, ultra thick, ->] (1.2,0) to (.6,-.6);
	\draw[black, ultra thick] (.6,-.6) to (0,-1.2);
	\fill[black] (0,1.2) circle (.1);
	\fill[black] (0,-1.2) circle (.1);
	\fill[black] (1.2,0) circle (.1);
	\fill[black] (-1.2,0) circle (.1);
	\end{tikzpicture}
	$}
};
(0,-15)*+{(G, \Oc)};
\end{xy}
\ \ \ \ \ \ \ \ 
\!\!\!\!\!\!\!\!\!\!\!\!\!\!\!\!\!\!\!
$$
\caption{A Metric graph $\Gamma$, a model $G$ with $\ell \equiv 1$ and an orientation $\Oc$.}\label{fig:models.ex}
\end{figure}
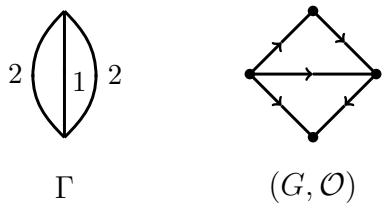

Let $\Gamma$ be a metric graph and fix $x, y \in \Gamma$. We call $\gamma$ a {\em (piecewise linear) path} from $x$ to $y$ if there exists a model $G$ for $\Gamma$ such that $x, y \in V(G)$, and $\gamma$ is represented by a walk on $G$. A path from $x$ to $x$ will be called {\em closed}.

\begin{Remark}
	A weighted graph can be naturally thought of as an {\em electrical network}, where each line segment of length $\ell(e)$ is thought of as a {\em resistor} with resistance $\ell(e)$.
\end{Remark}

\begin{Definition}
\begin{itemize}
\item[]
\item[(i)] The {\em genus} $g(\Gamma)$ of a metric graph $\Gamma$ is its topological genus, i.e. the dimension of its first homology or the rank of its fundamental group. 
\item[(ii)] The {\em Euler characteristic} of a metric graph of $\Gamma$ is its topological Euler characteristic, i.e. $\chi(\Gamma) = 1-g(\Gamma)$.
\end{itemize}
\end{Definition}
If $\Gamma$ has a {\em finite} model $G$, then the genus of $\Gamma$ is simply $g(\Gamma)=|E(G)|-|V(G)|+1$.

\subsection{Galois covers}

\begin{Definition}
\begin{itemize}
\item[]
\item[(i)] Given any metric space $X$, by a \emph{covering space} we mean a metric space $Y$ together with a continuous map $Y\rightarrow X$, called the \emph{covering map}, such that for any point $p\in X$, there is an open neighborhood $U$ of $p$ whose preimage under the covering map consists of a disjoint union of open sets that are sent to $U$ {\em isometrically} by the covering map. 
\item[(ii)] The {\em degree} of the covering map is the number of these disjoint copies.
\end{itemize}
\end{Definition}

\begin{Remark}
Any path-connected covering space of a metric graph is also a metric graph. 
\end{Remark}

 \begin{Definition}
 When both $X$ and $Y$ are path-connected, a covering map $\phi \colon Y\rightarrow X$ induces a map $\phi_* \colon \pi_1(Y,y)\rightarrow\pi_1(X,x)$, where $\phi(y)=x$. If $\phi_*\left((\pi_1(Y,y)\right)$ is a {\em normal} subgroup of $\pi_1(X,x)$, we call $Y$ a \emph{Galois} cover of $X$. 
 \end{Definition}

Alternatively, a covering $Y\rightarrow X$ between path-connected spaces is \emph{Galois} if and only if there is a free action of a discrete group $\Lambda$ on $Y$ such that $Y/\Lambda$ is isometric to $X$.

\section{Functions, forms, and measures on metric graphs}
\label{sec:forms}

\subsection{Piecewise linear functions and piecewise constant forms}

 \subsubsection{Bilinear forms on chains and cochains}

Let $G$ be a model for the metric graph $\Gamma$. As usual, we let $C_1(G, \RR)$, $C^1(G, \RR)$, $H_1(G, \RR) \simeq H_1(\Gamma, \RR)$, and $H^1(G, \RR) \simeq H^1(\Gamma, \RR)$ denote the $1$-chains, $1$-cochains, first homology, and first cohomology with coefficients in $\RR$.
We fix an orientation $\Oc$ on $G$ so that we can represent the space of $1$-chains by
\[C_1(G,\RR) = \frac{\bigoplus_{e \in \mathbb{E}(G)} \RR e }{ \langle e+\bar{e} \colon e \in \Oc \rangle} \simeq \bigoplus_{e \in \Oc} \RR e \, .\]

 We define an inner product
\begin{equation}\label{eq:ChainBilinear}
 \langle \cdot , \cdot \rangle \colon C_1(G, \RR) \times C_1(G, \RR) \rightarrow \RR \end{equation}  by 
\[
\langle e, e' \rangle = 
\begin{cases}
\ell(e) &\text{, if } e =e'\\
0 &\text{, if } e \ne e' \\
\end{cases}
\]
for all $e,e'\in\Oc$.

The space of $1$-cochains is denoted by $C^1(G,\RR)$, which is the vector space of all real-valued functions on $\EE(G)$ that are symmetric in the sense that for $\omega \in C^1(G,\RR)$ we have $\omega(\bar{e}) = - \omega(e)$. 
Here $\omega(e)$ denotes the evaluation of the functional $\omega$ on $e \in C_1(G, \RR)$. For $e \in \Oc$, we denote by $de\in C^1(G,\RR)$ the functional defined by 
\[
(de)(e') = 
\begin{cases}
\ell(e) &\text{, if } e =e'\\
0 &\text{, if } e \ne e'\\
\end{cases}
\]
for all $e,e'\in\Oc$. 
Any element $\omega\in C^1(G,\RR)$ can be written as a (possibly infinite) sum \[\omega=\sum_{e\in \Oc}\omega_e \, de \, , \quad \text{where} \quad \omega_e={\omega(e)}/{\ell(e)} \,.\] 
We have a (partially defined) bilinear form 
\begin{equation}\label{eq:coChainBilinear}
\langle \cdot , \cdot \rangle \colon C^1(G, \RR) \times C^1(G, \RR) \rightarrow \RR
\end{equation} 
sending the pair 
\[\omega=\sum_{e\in \Oc}\omega_e \, de , \quad \omega'=\sum_{e\in \Oc}\omega'_e \, de 
\]
to 
\[
\langle \omega , \omega' \rangle= \sum_{e \in \Oc}{\omega_e \, \omega'_e \,\ell(e)}\, .
\]

\begin{Remark} \phantomsection \label{rmk:conv}
\begin{itemize}
\item[]
\item[(i)] If $G$ is an infinite graph, then the sum on the right in \eqref{eq:coChainBilinear} might not converge, or might be $\infty$.
\item[(ii)] The bilinear form \eqref{eq:coChainBilinear} on $C^1(G, \RR)$ coincides with the bilinear form obtained from \eqref{eq:ChainBilinear} on $C_1(G,\RR)$ by transport of structure.
\end{itemize}
\end{Remark}

\subsubsection{Piecewise constant forms}

Let $G$ be a weighted graph with a length function $\ell$. Assume $G'$ is the graph obtained by subdividing an edge $e \in E(G)$ into two edges $e_1$ and $e_2$ with $\ell(e_1)+\ell(e_2) = \ell(e)$ (see Figure~\ref{fig:subdivide}). For any $\omega \in C^1(G,\RR)$ there is a natural $\omega' \in C^1(G',\RR)$ defined by 
\begin{equation} \label{eq:subdivide}
\omega'_{e'} = 
\begin{cases}
\omega_e &\text{, if } {e'} =e_1\\
\omega_e &\text{, if } {e'}=e_2\\
\omega_{e'} &\text{, otherwise. } 
\end{cases}
\end{equation}

\begin{figure}
\begin{tikzpicture}[scale=2, ray/.style={decoration={markings,mark=at position .5 with {
      \arrow[>=latex]{>}}},postaction=decorate}
]
\draw[-](0,0)--(1,0);
\draw[-](2,0)--(3,0);
\draw[fill](0,0) circle (1pt);
\draw[fill](1,0) circle (1pt);
\draw[fill](2,0) circle (1pt);
\draw[fill](3,0) circle (1pt);
\draw[fill](2.3,0) circle(1pt);
\node at (.5,0.2){$e$};
\node at (2.15,0.2){$e_1$};
\node at (2.65,0.2){$e_2$};

\end{tikzpicture}

\caption{\label{fig:subdivide}Subdivision of an edge: $\ell(e) = \ell(e_1)+\ell(e_2)$}
\end{figure}
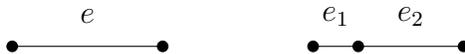 

The association $\omega \mapsto \omega'$ gives an injection $\iota \colon C^1(G,\RR) \hookrightarrow C^1(G',\RR)$ respecting the associated bilinear forms $\langle \cdot , \cdot \rangle$.

With this observation, we may consider the family $\{C^1(G,\RR)\}$, where $G$ varies over all models for the metric graph $\Gamma$, as a directed system. Naturally, we will look at the direct limit 
\begin{equation}\label{eq:forms}
 \Omega^1(\Gamma) = \varinjlim{C^1(G,\RR)}\, .
 \end{equation}
 
Because all the maps induced by refinements are injective, we may (and will) identify $C^1(G,\RR)$ with its image in $\Omega^1(\Gamma)$.
Since the bilinear form $\langle \cdot , \cdot \rangle$ is respected under $\iota$, we also obtain a bilinear form $\langle \cdot , \cdot \rangle$ on $\Omega^1(\Gamma)$. 

\begin{Definition}
\begin{itemize}
\item[]
\item[(i)] We call an element of $\Omega^1(\Gamma)$ a {\em piecewise constant 1-form} (or just a {\em form}) on $\Gamma$
\item[(ii)] A form $\omega \in \Omega^1(\Gamma)$ that can be represented by an element in $C^1(G,\RR)$ is said to be {\em compatible} with the model $G$.
\end{itemize}
\end{Definition}

\subsubsection{$L^2$ spaces}
If the metric graph $\Gamma$ is not compact, then any model $G$ for $\Gamma$ must be an infinite graph.
  As noted in Remark \ref{rmk:conv}~(i), the bilinear pairing on $C^1(G,\RR)$ (hence on $H^1(G,\RR)$ and on $\Omega^1(\Gamma)$) is only partially defined.

\begin{Definition}
\begin{itemize}
\item[]
\item[(i)] We call an element of 
\[\Omega^1_{L^2}(\Gamma) = \{\omega\in \Omega^1(\Gamma) \colon \langle\omega,\omega\rangle<\infty\} \]
an {\em $L^2$ piecewise constant 1-form} (or just an {\em $L^2$ form}).
\item[(ii)] We call an element of 
\[C^1_{L^2}(G,\RR) = \{\omega\in C^1(G, \RR) \colon \langle\omega,\omega\rangle<\infty\} \]
an {\em $L^2$ cochain}.
\end{itemize}
\end{Definition}
For any model $G$, we will identify $C^1_{L^2}(G,\RR)$ with its image in $\Omega^1_{L^2}(\Gamma)$. The bilinear space $\left( C^1_{L^2}(G,\RR) , \langle \cdot , \cdot \rangle \right)$ is an $L^2$ Hilbert space (with respect to the measure $\nu_\ell$ defined by $\nu_\ell(\{e\})=\ell(e)$). However, $\left(  \Omega^1_{L^2}(\Gamma)  , \langle \cdot , \cdot \rangle \right)$ is only a pre-Hilbert space.

\subsubsection{Piecewise linear functions, derivatives, and integrals}

\begin{Definition}
Let $\Gamma$ be a metric graph. A continuous function $f \colon \Gamma \rightarrow \RR$ is called \emph{piecewise linear} if there is a model $G$ for $\Gamma$ such that the restriction of $f$ to each edge of $G$ is linear. We call such a model $G$ {\em compatible} with $f$. 
The collection of all piecewise linear continuous functions on $\Gamma$ will be denoted by $\PL(\Gamma)$.
\end{Definition}

Consider a function $f \in \PL(\Gamma)$ with a compatible model $G$. The restriction of $f$ to $V(G)$ gives an element in $C^0(G,\RR)$. Its image under the coboundary map will be denoted by $df \in  C^1(G,\RR)$. 
In other words,
\[df = \sum_{e \in \Oc} (df)_e \, de \, ,\] 
where $(df)_e$ is the slope of $f$ on the oriented edge $e$:  
\[
(df)_e = 
\frac{f(e^+)-f(e^-)}{\ell(e)} \, .
\]
 By \eqref{eq:forms}, we have the canonical (injective) function $\phi_G \colon C^1(G,\RR) \rightarrow \Omega^1(\Gamma)$ sending each element to its equivalence class. In this way, we obtain a map
\begin{equation}\label{eq:d.R.Omega}
d \colon \PL(\Gamma) \rightarrow \Omega^1(\Gamma)
\end{equation}
sending $f$ to $\phi_G(df)$. We refer to $df$ as the {\em derivative} of $f$. The kernel of $d$ consists of constant functions on $\Gamma$.

Let $\gamma$ be a piecewise linear path in $\Gamma$, and let $G$ be a model such that $\gamma$ is represented by a walk $\beta$ in $G$:
\[v_0 , e_0 , v_1 , e_1 , v_2 , \ldots , v_{k-1} , e_{k-1} , v_k \, .\]

For this path $\gamma$, we have an associated $1$-cochain $\omega_\gamma = \sum_{i=0}^{k-1}{de_i}$. The equivalence class of $\omega_\gamma$ in $\Omega^1(\Gamma)$ will also be denoted by $\omega_\gamma$, and will be referred to as the {\em 1-form associated to $\gamma$}.

\begin{Definition}
Let $\gamma$ be a path on $\Gamma$ and $\omega \in \Omega^1(\Gamma)$. We define {\em the integral of $\omega$ along $\gamma$} by 
\[\int_\gamma\omega =  \langle \omega_\gamma,\omega\rangle \, .\]
\end{Definition}

We have the following metric graph analogue of (the converse of) the {\em Gradient theorem} or the {\em Poincar\'e lemma}.
\begin{Lemma}\label{lem:Gradient}
Let $\Gamma$ be a metric graph, and let $\omega \in \Omega^1(\Gamma)$. 
Assume
\begin{equation}\label{eq:intomega}
\int_\gamma\omega = 0 \, ,
\end{equation} 
for all {\em closed} paths $\gamma$. 
Then $\omega = df$ for some $f \in \PL(\Gamma)$.

More precisely, fix a point $z \in \Gamma$, and for any $x \in \Gamma$ let $\gamma_{zx}$ be a path from $z$ to $x$. Then the function $f\in \PL(\Gamma)$ given by
\begin{equation}\label{int}
f(x)=\int_{\gamma_{zx}}\omega
\end{equation}
is well-defined and 
$ \omega=df$.
\end{Lemma}
\begin{Remark}
Since the kernel of $d$ consists of constant functions, any other such $f$ differs from \eqref{int} by a constant function. The choice of the constant is fixed by the choice of the base point $z \in \Gamma$.
\end{Remark}

\begin{proof} 
The function $f$ in \eqref{int} is well-defined (independent of the choice of the path $\gamma_{zx}$) because any two different paths differ by a closed path. 

Next we show $\omega = df$, for $f$ as in \eqref{int}. Let $G$ be a model for $\Gamma$ compatible with $\omega$ such that $z\in V(G)$. Let $e$ be an arbitrary oriented edge of $G$. For any $x\in e$ consider the model $G_x$ whose vertex set is $V(G)\cup\{x\}$. The refinement from $G$ to $G_x$ is the decomposition of $e$ into two edges $e_1 = \{e^- , x\}$ and $e_2 = \{x , e^+\}$. By definition (see \eqref{eq:subdivide}) $\omega_{e_1}=\omega_{e_2}=\omega_e$. Let $\gamma_0$ be a path from $z$ to $e^-$. Then, a path $\gamma_{zx}$ from $z$ to $x$ can be chosen to be the concatenation of $\gamma_0$ and $e_1$. Then
\[\int_{\gamma_{zx}}\omega=\langle \omega_{\gamma_{zx}},\omega\rangle=\langle \omega_{\gamma_0}+de_1,\omega\rangle=\langle\omega_{\gamma_0},\omega\rangle+\ell(e_1) \, \omega_e \, .\]
This is a linear function of $\ell(e_1)$ with slope $\omega_e$. Hence, $(df)_{e}=\omega_e$. This proves the claim, since $e$ was an arbitrary edge.
 \end{proof}

\subsection{Harmonic forms and Laplacians}
\subsubsection{Harmonic forms}
Let $\Gamma$ be a metric graph and let $G$ be a model for $\Gamma$. Let $C^0_c(G,\RR) \subseteq C^0(G,\RR)$ be the space of $0$-cochain with compact support: for each $\phi \in C^0_c(G,\RR)$ there exists some compact $K \subseteq \Gamma$ such that  $\phi$ vanishes on all $0$-chains supported on $\Gamma \backslash K$. We have already observed that $C^0(G,\RR)$ may be identified with those functions in $\PL(\Gamma)$ that are compatible with the model $G$. The coboundary map $d \colon C^0(G,\RR)\rightarrow C^1(G,\RR)$ induces a map \[d^*\colon C^1(G,\RR)\rightarrow (C^0_c(G,\RR))^*\] defined by
\begin{equation} \label{eq:dstar}
(d^*\omega)(\phi)=\langle\omega,d\phi\rangle \, ,
\end{equation}
where $\langle\cdot,\cdot\rangle$ is the bilinear pairing \eqref{eq:coChainBilinear}, and $\phi$ is a piecewise linear function with compact support and compatible with $G$. Note that, since $\phi$ has compact support, the bilinear product $\langle\omega,d\phi\rangle$ is a well-defined. 

We may also identify an element $\beta\in (C^0_c(G,\RR))^*$ with a locally finite signed Borel measures on $\Gamma$
\[
\beta = \sum_{v\in V(G)}\beta(v) \delta_v \, ,
\] 
where $\delta_v$ denotes the Dirac measure at $v$. Clearly, $\Supp(\beta) \subseteq V(G)$, where $\Supp(\beta)$ denotes the {\em support} of the measure $\beta$. 

With this notation, we may explicitly write, for $\omega = \sum_{e \in \Oc}{\omega_e \ de} \in C^1(G,\RR)$,
\begin{equation}\label{adj}
d^* \omega=\sum_{v\in V(G)} \left(\sum_{e^+ = v} {\omega_e} - \sum_{e^- = v} {\omega_e} \right)\delta_v \, .
\end{equation}

Let $\Meas(\Gamma)$ denote the the space of all locally finite signed Borel measures on $\Gamma$. It follows from \eqref{adj} that, for a vertex $v \in V(G)$ of valence $2$, we have $(d^*\omega)(v)=0$ if and only if $\omega$ is constant around $v$. Therefore, we obtain a well-defined map
\begin{equation} \label{eq:d-star}
d^* \colon \Omega^1(\Gamma) \rightarrow \Meas(\Gamma) \, . 
\end{equation} 

For any $\omega \in \Omega^1(\Gamma)$, a model $G$ is compatible with $\omega$ if and only if $\Supp(d^*\omega) \subseteq V(G)$.

\begin{Definition}
Given a subset $A\subseteq\Gamma$, we call a form $\omega\in\Omega^1(\Gamma)$ {\em harmonic} on $A$ if 
\[ 
A \cap \Supp(d^*\omega) = \emptyset \, .
\] 
\end{Definition}\label{def:har}

We have the space of (global) harmonic forms on $\Gamma$:
\[
\mathcal{H}(\Gamma) = \{\omega \in \Omega^1(\Gamma) \colon d^*\omega = 0 \} \, .
\] 
Furthermore, the space of $L^2$ (global) harmonic forms on $\Gamma$ is:
 \[
 \mathcal{H}_{L^2}(\Gamma) = \{\omega \in \Omega^1_{L^2}(\Gamma) \colon d^*\omega = 0 \} = \mathcal{H}(\Gamma) \cap \Omega^1_{L^2}(\Gamma) \, .
 \]

The following straightforward result on harmonic forms is quite useful:

\begin{Lemma}\label{lemma:bridge}
	Let $G$ be a model for a metric graph $\Gamma$, and let $\omega\in\mathcal{H}(\Gamma)$. 
Assume $e\in E(G)$ is a bridge and (at least) one of the two connected components $\Gamma_1$ and $\Gamma_2$ of $\Gamma \backslash e$ is compact (see Figure~\ref{fig:comp}). 
Let $\omega^{(i)} = \omega|_{\Gamma_i}$ denote the restriction of $\omega$ to $\Gamma_i$ (for $i = 1,2$). 
Then 
\begin{itemize} 
\item[(a)] $\omega_e=0$.
\item[(b)] $\omega^{(1)}, \omega^{(2)} \in \mathcal{H}(\Gamma)$.
\item[(c)] The contraction map $\Gamma \rightarrow \Gamma / e$ induces an isomorphism $\mathcal{H}(\Gamma) \xrightarrow{\sim} \mathcal{H}(\Gamma / e)$. 
	\end{itemize} 
\end{Lemma}

\begin{proof}
  
\vspace{-2mm}
\begin{figure}[H]
\begin{tikzpicture}[scale=.75]
\draw[black, ultra thick, <-](.5,0)--(1,0);
\draw[black, ultra thick, -](0,0)--(.5,0);
\draw[dotted](-1,0) circle (1);
\node at (.5,0.4){$e$};
\draw[dotted](2,0) circle (1);
\node at (-1,0){$\Gamma_1$};
\node at (+2,0){$\Gamma_2$};
\end{tikzpicture}
\caption{$e$ is a bridge and $\Gamma_1$ or $\Gamma_2$ is compact.}\label{fig:comp}
\end{figure}
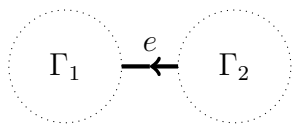 

	We assume $\Gamma_1$ is compact. Let $G_1\subseteq G$ be the finite subgraph of $G$ that is a model for $\Gamma_1$. Fix an orientation $\Oc$ for $G$ such that $e$ is directed towards $G_1$. By \eqref{adj}, we have 
	\[
	\sum_{e'\in \Oc,e'^+=v}\omega_{e'}-\sum_{e'\in \Oc,e'^-=v}\omega_{e'} = 0 \text{ , for all } v \in V(G)\,  .
	\]
	Summing over all (finitely many) vertices in $V(G_1)$ we obtain
	\[
	0=\sum_{v\in V(G_1)}\left(\sum_{e'\in \Oc,e'^+=v}\omega_{e'}-\sum_{e'\in \Oc,e'^-=v}\omega_{e'}\right)=\omega_e\, .
	\]
	This proves (a). 
	Part (b) follows from (a) and \eqref{adj}. 
	Part (c) follows from (a) and (b).
\end{proof}

\subsubsection{Hodge decomposition}

\begin{Proposition}\label{prop:df w orth}
Let $\Gamma$ be a metric graph, $G$ be any model for $\Gamma$. Then:
\begin{itemize}
\item[(a)] The space $\mathcal{H}_{L^2}(\Gamma)$ is a closed subspace of the Hilbert space $C^1_{L^2}(G,\RR)$.
\item[(b)] The orthogonal complement of $\mathcal{H}_{L^2}(\Gamma)$ in $C^1_{L^2}(G,\RR)$ is contained in 
\[{\rm Image}\left(d \colon \PL(\Gamma) \rightarrow \Omega^1(\Gamma) \right) \cap C^1_{L^2}(G,\RR) \, .
\]  
\item[(c)] The orthogonal complement of $\mathcal{H}_{L^2}(\Gamma)$  in $C^1_{L^2}(G,\RR)$ is $\cl\left(d(C^0_c(G,\RR))\right)$. 
\end{itemize}
\end{Proposition}

\begin{proof} For any $\omega\in\mathcal{H}_{L^2}(\Gamma)$, since $\Supp(d^*\omega)=\emptyset$, we know $\omega$ must be compatible with any model $G$ for $\Gamma$, so $\mathcal{H}_{L^2}(\Gamma)\subseteq C^1_{L^2}(G,\RR) \subseteq \Omega^1_{L^2}(\Gamma)$.
It follows
\[
\begin{aligned}
\mathcal{H}_{L^2}(\Gamma)=&\{\omega\in C^1_{L^2}(G,\RR) \colon d^*\omega=0\}\\
=&\{\omega\in C^1_{L^2}(G,\RR) \colon \langle \omega,df\rangle=0, \forall f\in C^0_c(G,\RR)\}\\
=&\left(d(C^0_c(G,\RR))\right)^\perp  \, ,
\end{aligned}
\]
where $(\cdot)^\perp$ denotes the orthogonal complement. As a consequence, $\mathcal{H}_{L^2}(\Gamma)$ is closed because any orthogonal complement in a Hilbert space is closed, which proves (a). Part (c) follows from the fact that for any subspace $V$ of a Hilbert space we have $(V^\perp)^\perp=\cl(V)$.

It follows from equation \eqref{adj} that, for any {\em closed} path $\gamma$ on $\Gamma$, we have $\omega_\gamma\in\mathcal{H}(\Gamma)$. Because $\omega_\gamma$ is  supported on $\gamma$ which is a compact set, we further have $\omega_\gamma \in \mathcal{H}_{L^2}(\Gamma)$. As a consequence, if $\alpha\in C^1_{L^2}(G,\RR)$ is orthogonal to $\mathcal{H}_{L^2}(\Gamma)$ then $\alpha$ is orthogonal to all the 1-forms $\omega_\gamma$ corresponding to closed paths $\gamma$ in $\Gamma$. Part (b) now follows from Lemma \ref{lem:Gradient}. 
\end{proof}

For {\em compact} metric graphs, we have the following version of ``Hodge decomposition''.
\begin{Proposition}\label{prop:hodge}
	Let $\Gamma$ be a {\em compact} metric graph.
	\begin{itemize}
		\item[(a)] There is a direct sum decomposition $\Omega^1(\Gamma)=d \PL(\Gamma)\oplus\mathcal{H}(\Gamma)$, which is an orthogonal decomposition under the bilinear pairing on $\Omega^1(\Gamma)$.
		\item[(b)] The composition 
\begin{equation} \label{harmonic vs cohomology}
\mathcal{H}(\Gamma) \hookrightarrow C^1(G , \RR)  \twoheadrightarrow H^1(\Gamma,\RR)
\end{equation}
		is an isomorphism.
	\end{itemize}
\end{Proposition} 
\begin{proof}
	Because $\Gamma$ is compact, it has a finite model $G$ and $C^1_{L^2}(G,\RR)=C^1(G,\RR)$, and $\mathcal{H}_{L^2}(\Gamma)=\mathcal{H}(\Gamma)$. Because $\mathcal{H}(\Gamma)$ is a subspace of a finite dimensional vector space $C^1(G,\RR)$, it is of finite dimension and is closed and complete under the $L^2$ norm. Furthermore, all piecewise linear functions are compactly supported, hence it follows from Proposition~\ref{prop:df w orth} that the orthogonal complement of $\mathcal{H}(\Gamma)$ in $C^1(G,\RR)$ is
\begin{equation} \label{eq:HorthD}
(\mathcal{H}(\Gamma))^\perp=d(C^0(G,\RR)) \, .
\end{equation}
 
	Let $\omega \in \Omega^1(\Gamma)$ be in the orthogonal complement of $\mathcal{H}(\Gamma)$. Then $\omega \in C^1(G',\RR)$ for {\em some} model $G'$. It follows from \eqref{eq:HorthD} that $\omega \in d(C^0(G',\RR))$. Therefore $\omega \in d \PL(\Gamma)$. This proves part (a).
	
	Part (b) follows from $H^1(\Gamma, \RR) \simeq C^1(G,\RR)/d\left(C^0(G,\RR)\right)$ and \eqref{eq:HorthD}.
\end{proof}

\begin{Definition}
Let $\Gamma$ be any metric graph. The operator 
\[
\Delta=d^* \circ d \colon \PL(\Gamma) \rightarrow \Meas(\Gamma) 
\] 
(the composition of maps \eqref{eq:d.R.Omega} and \eqref{eq:d-star}) is called the {\em Laplacian operator}. 
\end{Definition}
Alternatively, $\Delta$ is the usual Laplacian operator in the sense of {\em distributions}. More concretely, the Laplacian $\Delta(f)$, for $f \in \PL(\Gamma)$, is the discrete measure 
\begin{equation} \label{eq:laplace}
\sum_{p \in \Gamma} {\sigma_p(f)\, \delta_p} \, ,
\end{equation}
where $\sigma_p(f)$ is the sum of incoming slopes of $f$ at $p$.

\begin{Definition}\label{def:har2}
We call $f\in \PL(\Gamma)$ {\em harmonic} on $A$ if $A \cap \Supp(\Delta f) = \emptyset$.
\end{Definition}
It follows from Definition \ref{def:har} and Definition \ref{def:har2} that $f\in \PL(\Gamma)$ is a harmonic function on $A$ if and only if $df \in\Omega^1(\Gamma)$ is a harmonic form on $A$.

\begin{Lemma}\label{lem:maxprin}
Let $\Gamma$ be a metric graph and $f \in \PL(\Gamma)$.
\begin{itemize}
\item[(a)] (Maximum principle) If $f$ is harmonic on some open set $U$, and is not locally constant, then $f$ cannot attain its maximum (or minimum) in the interior of $U$. 
\item[(b)] If $\Gamma$ is compact, then any $f\in \PL(\Gamma)$ which is harmonic on $\Gamma$ must be a constant. 
\item[(c)] If $f \in \PL(\Gamma)$ is harmonic on $U$ and $p \in U$ has valence $2$, then $f$ must be linear on a neighborhood of $p$.
\end{itemize}
\end{Lemma}

\begin{proof}
(a) and (c) follow immediately from \eqref{eq:laplace}.
(b) follows from (a). 
\end{proof}

\subsection{Fundamental kernel of the Laplacian}
For {\em compact} metric graphs, a fundamental solution of the Laplacian is given by $j$-functions. We follow the notation of \cite{chinburg1993capacity} (see also \cite{zhang1993admissible, br, baker2007potential, Sho10, BS13}).  The main result in this section is an explicit formula in Proposition \ref{prop:jexists} for $j_z(x,y)$, which we will need later and it might also be of independent interest. 

\begin{Definition}\label{def:jfunction}
Let $\Gamma$ be a compact metric graph and fix two points $y,z \in \Gamma$. We denote by $j_z(\cdot, y)$ or $j^{\Gamma}_z(\cdot,y)$ the unique function in $\PL(\Gamma)$ satisfying:
\begin{itemize}
	\item[(i)] $\Delta_x{j_z(x,y)} = \delta_y - \delta_z$, 
	\item[(ii)] $j_z(z,y) = 0$.
\end{itemize}
\end{Definition}

\begin{Remark}\label{rmk:elect} 
The function $j_z(x, y)$ has a nice electrical network interpretation: it denotes the electric potential at $y$ if one unit of current enters the network at $x$ and exits at $z$, with $z$ ``grounded'' (i.e. has zero potential). So, the existence and uniqueness of $j_z(\cdot, y) \in \PL(\Gamma)$ are well known to electrical engineers.
\end{Remark}

For any compact metric graph $\Gamma$, the fundamental kernel $j_z(\cdot,y) \in \PL(\Gamma)$ exists and is unique.
The uniqueness of $j_z(\cdot,y)$ follows from Lemma \ref{lem:maxprin}~(b); if $j$ and $j'$ are two functions satisfying Definition~\ref{def:jfunction}, $F(\cdot)=j_z(\cdot,y)-j'_z(\cdot,y)$ is harmonic on $\Gamma$, so $F$ must be a constant. Furthermore we have $F=F(z)=j_z(z,x)-j'_z(z,x)=0$. 
The existence of $j_z(\cdot,y)$ follows from the explicit formula in Proposition~\ref{prop:jexists}.

\begin{Proposition}\label{prop:jexists}	
Let $\Gamma$ be a compact metric graph. Let 
	\[
	\pi \colon \Omega^1(\Gamma) \rightarrow \mathcal{H}(\Gamma)
	\] 
	be the orthogonal projection. Then for all $x,y,z \in \Gamma$ we have 
	\begin{equation}\label{eq:jformula}
	j_z(x,y)=\int_{\gamma_{zx}}(\omega_{\gamma_{zy}}-\pi(\omega_{\gamma_{zy}})) \, ,
	\end{equation}
	where $\gamma_{uv}$ denotes a path from $u$ to $v$.
\end{Proposition}

\begin{proof}
First, we will show that \eqref{eq:jformula} is well-defined, i.e. different choices of paths $\gamma_{zx}$ and $\gamma_{zy}$ does not change the value of the integral.  
Different paths $\gamma_{zy}$ will differ by a closed path. Therefore, since $\pi$ is identity on $\mathcal{H}(\Gamma)$, the form $\omega_{\gamma_{zy}}-\pi(\omega_{\gamma_{zy}})$ is independent of the choice of $\gamma_{zy}$.
For any path $\gamma$, we have $\omega_{\gamma}-\pi(\omega_{\gamma}) \in (\mathcal{H}(\Gamma))^{\perp}$ because
\[
\pi\left(\omega_{\gamma}-\pi(\omega_{\gamma})\right) = \pi(\omega_{\gamma}) - \pi^2(\omega_{\gamma}) = \pi(\omega_{\gamma})-\pi(\omega_{\gamma}) = 0 \, .
\]
Therefore \eqref{eq:jformula} does not depend on the choice of the path $\gamma_{zx}$ either. 

Next, by Lemma \ref{lem:Gradient}, we know the derivative (with respect to $x$):
\[
d \left (\int_{\gamma_{zx}}(\omega_{\gamma_{zy}}-\pi(\omega_{\gamma_{zy}})) \right) = \omega_{\gamma_{zy}}-\pi(\omega_{\gamma_{zy}}) \,  .
\]
Therefore 
\[
\Delta \left (\int_{\gamma_{zx}}(\omega_{\gamma_{zy}}-\pi(\omega_{\gamma_{zy}})) \right) = d^*\left(\omega_{\gamma_{zy}}-\pi(\omega_{\gamma_{zy}})\right) = \delta_y - \delta_z\,  ,
\]
because $d^*\left(\omega_{\gamma_{zy}}\right) =  \delta_y - \delta_z$ by \eqref{adj}, and $d^*\left(\pi(\omega_{\gamma_{zy}})\right) = 0$ since $\pi(\omega_{\gamma_{zy}})$ is harmonic.

If $z=x$ then the constant path gives $\omega_{\gamma_{zx}} = 0$ and therefore $j_z(z,y) = 0$.
\end{proof}

\subsection{Measures and their convergence} 

\subsubsection{Piecewise Lebesgue measures}
\begin{Definition}
Let $\Gamma$ be a metric graph. We will call a real-valued Borel measure $\mu$ on $\Gamma$ {\em piecewise Lebesgue}, if 
\[\mu = fdx \, ,\] 
where $dx$ is the Lebesgue measure when restricted to any interval embedded in $\Gamma$, and $f$ is piecewise constant (i.e. locally constant on $\Gamma$ except for a discrete set). 
\end{Definition}

Let $G$ be a model for $\Gamma$ which is {\em compatible} with $\mu$, meaning $V(G)$ includes all the points in $\Gamma$ such that $f$ is not continuous. Then the piecewise Lebesgue measure $\mu$ is uniquely determined by its values $\{\mu(e)\}_{e \in E(G)}$ because
\[
\mu|_e = \frac{\mu(e)}{\ell(e)}dx \, .
\]

\subsubsection{Convergence of measures} \label{sec:convmeas}
The definition of the convergence of a sequence of measures $\{\mu_n\}$ is convergence of all integrals against a given continuous compactly supported function.
When we restrict to measures on $\Gamma$ which are compatible with a fixed model $G$, this convergence is equivalent to the convergence of the sequences $\{\mu_n(e)\}$ for all $e\in E(G)$. 

When $G$ is finite ($\Gamma$ is compact) the space of signed measures compatible with $G$ is a finite dimensional vector space (because each measure is determined by the evaluation on $E(G)$), hence the above convergence coincides with the notions of strong convergence and weak convergence.

\subsubsection{Pullback and pushdown of measures}
\label{sec:pullback}
Let $\phi:\Gamma'\rightarrow\Gamma$ be a continuous map that is simplicial for models $G'$ for $\Gamma'$ and $G$ for $\Gamma$. After possible refinement, we can assume that it send each edge in $G'$ to a vertex or an edge in $G$. 

Let $\mu$ be a piecewise Lebesgue measure on $\Gamma$ compatible with $G$. Then the {\em pullback} of $\mu$ is a measure on $\Gamma'$ compatible with $G'$ defined by $(\phi^*\mu)(e)=\mu(\phi(e))$.

If $\phi \colon \Gamma'\rightarrow\Gamma$ is a Galois covering, pullback gives a bijection between the set of measures on $\Gamma$ and the set of measures on $\Gamma'$ which are invariant under deck transformations. If $\mu'$ is a measure on $\Gamma'$ which is invariant under the deck transformations, we will call the corresponding measure $\mu$ on $\Gamma$ under this bijection the {\em pushdown} of $\mu'$.

\begin{Remark}
Our notion of pushdown is inspired by the concept of the pushdown of metrics, and is different from the notion of pushforward in measure theory; we do not record the total measure of the preimage, we only record the measure of one connected component of the preimage. So, in the case of a finite Galois cover of degree $d$, the pushforward of a measure $\mu'$ is $d$ times the pushdown of $\mu'$. 
\end{Remark}

\section{Hilbert $\G$-modules and $\G$-dimensions}\label{sec:Gdim}
\label{sec:gdim}
Here we present the von Neumann algebras and dimensions that appear in our context. See \cite{luck} and \cite{pansu1993introduction} for proofs and a more thorough treatment. The computations appearing in Example~\ref{eg:C1dim} will be used in the proof of Theorem~\ref{thm:GaussBonnet}.

By a Hilbert $\G$-module we mean a Hilbert space $\H$ together with a (left) unitary action of a discrete group $\G$. For any Hilbert space $\H$, let $\B(\H)$ denote the algebra of all bounded linear operators on $\H$. Let 
\[
\B(\H)^+ = \{A \in \B(\H) \colon \langle Ax , x \rangle \in \RR^{\geq 0} , \, \forall x \in \H \} \, .
\]

\begin{Definition}
A {\em free} Hilbert $\G$-module is a Hilbert $\G$-module which is unitarily isomorphic to $\ell^2(\G) \otimes \H$, where $\H$ is a Hilbert space with the trivial $\G$-action and the action of $\G$ on $\ell^2(\G)$ is by left translations. In other words, the representation of $\G$ on $\ell^2(\G) \otimes \H$ is given by $g \mapsto L_g\otimes I$.

Recall, for each $g \in \G$, we have the left and right translation operators $L_g , R_g \in \B(\ell^2(\G))$ defined by:
\[
L_g(f)(h) = f(g^{-1}h) \quad , \quad R_g(f)(h) = f(hg) \quad , \quad \text{for } h \in \G \, .
\]

\end{Definition}
Let $\{u_\alpha\}_{\alpha \in J}$ be an orthonormal basis for $\H$. Then we have an orthogonal decomposition
\[
\ell^2(\G) \otimes \H = \bigoplus_{\alpha \in J} \ell^2(\G)^{(\alpha)} = \{\sum_{\alpha \in J}{f_{\alpha} \otimes u_{\alpha}} \colon f_{\alpha} \in \ell^2(\G) , \,  \sum_{\alpha \in J} \|f_{\alpha}\|^2 < +\infty \} \, , 
\]
where $\ell^2(\G)^{(\alpha)} = \ell^2(\G) \otimes u_{\alpha}$ is just a copy of $\ell^2(\G)$. Moreover, an orthonormal basis for $\ell^2(\G) \otimes \H$ is $\{\delta_g \otimes u_{\alpha} \colon g \in \G , \, \alpha \in J\}$.

We are interested in the von Neumann algebra $\M_r(\G) \otimes \B(\H)$ on $\ell^2(\G) \otimes \H$, where $\M_r = \M_r(\G)$ is the von Neumann algebra generated by $\{R_g \colon g \in \G \} \subseteq \B(\ell^2(\G))$. Alternatively $\M_r(\G)$ is the algebra of $\G$-equivariant bounded operators on $\ell^2(\G)$.

\begin{Definition} \label{def:trG}
Let $\{u_\alpha\}_{\alpha \in J}$ be an orthonormal basis for $\H$.
For every $A \in (\M_r(\G) \otimes \B(\H))^+$, define 
\[
\Tr_{\G}(A)  = \sum_{\alpha \in J} \langle A(\delta_h \otimes u_{\alpha}) , \delta_h \otimes u_{\alpha} \rangle \,  ,
\]
for $h \in \G$. This is independent of the choice of $h$, so it is convenient to use $h = {\rm id}$, the group identity. It can be checked that this is a ``trace function'' in the sense of von Neumann algebras.

\end{Definition}

\begin{Definition}
A projective Hilbert $\G$-module is a Hilbert $\G$-module $\V$ which is unitarily isomorphic to a closed submodule of a free Hilbert $\G$-module, i.e. a closed $\G$-invariant subspace in {\em some} $\ell^2(\G) \otimes \H$.
\end{Definition}

Note that the embedding of $\V$ into $\ell^2(\G) \otimes \H$ is {\em not} part of the structure; only its existence is required.

For the moment, we will fix such an embedding. Let $P_{\V}$ denote the orthogonal projection from $\ell^2(\G) \otimes \H$ onto $\V$. Then $P_{\V} \in \M_r(\G) \otimes \B(\H)$ because it commutes with all $L_g \otimes I$.

\begin{Definition}\label{def:dimG}
Let $\V$ be a projective Hilbert $\G$-module. Fix an embedding into a free Hilbert $\G$-module $\ell^2(\G) \otimes \H$, and let $P_{\V} \in \M_r(\G) \otimes \B(\H)$ denote the orthogonal projection onto $\V$. 
The $\G$-dimension of $\V$ is defined as
\[
\dim_{\G}(\V) = \Tr_{\G}(P_{\V}) \, .
\]
\end{Definition}
An elementary fact is that $\dim_{\G}(\V)$ does {\em not} depend on the choice of the embedding of $\V$ into a free Hilbert $\G$-module. Therefore it is a well-defined invariant of the projective Hilbert $\G$-module $\V$. 

\begin{Remark} \label{rmk:dimGprops}
$\G$-dimension satisfies the following properties:
\begin{itemize}
\item[(i)] $\dim_{\G}(\V) = 0 \iff \V = \{0\}$.
\item[(ii)] $\dim_{\G}(\ell^2(\G)) = 1$.
\item[(iii)] $\dim_{\G}(\ell^2(\G) \otimes H) = \dim_{\CC}(H)$.
\item[(iv)] $\dim_{\G}(\V_1\oplus \V_2)=\dim_{\G}(\V_1) + \dim_{\G}(\V_2)$.
\item[(v)] $ \V_1 \subseteq \V_2 \implies \dim_{\G}(\V_1) \leq \dim_{\G}(\V_2)$. Equality holds if and only if $\V_1 = \V_2$.
\item[(vi)] If $0\rightarrow \mathrm{U} \rightarrow \V \rightarrow \mathrm{W} \rightarrow 0$ is a short {\em weakly exact sequence} of projective Hilbert $\G$-modules, then $\dim_{\G}(\V) = \dim_{\G}(\mathrm{U})+\dim_{\G}(\mathrm{W})$.
\item[(vii)] $\V$ and ${\rm W}$ are {\em weakly isomorphic}, then $\dim_{G}(\V) = \dim_{G}({\rm W})$.
\end{itemize}
A sequence of $\mathrm{U} \xrightarrow{i} \V \xrightarrow{p} \mathrm{W}$ of projective Hilbert $\G$-modules is called {\em weakly exact} at $\V$ if ${\rm Kernel}(p) =  \cl({\rm Image}(i))$. A map of projective Hilbert $\G$-modules $\V \rightarrow {\rm W}$ is a {\em weak isomorphism} if it is injective and has dense image.
\end{Remark}
The main application of this general theory to our work is related to the following example, which will be used in the proof of Theorem~\ref{thm:GaussBonnet}.
\begin{Example} \label{eg:C1dim}
 Let $\phi \colon \Gamma' \rightarrow \Gamma$ be an infinite Galois covering of a compact metric graph $\Gamma$. Fix models $G$ and $G'$ for $\Gamma$ and $\Gamma'$, compatible with the covering map $\phi$. Let $\G = \pi_1(\Gamma) / \pi_1(\Gamma')$ be the deck transformation group. We claim 
 \[
 C^1_{L^2}(G', \CC) = C^1_{L^2}(G', \RR) \otimes_{\RR} \CC
 \]
 is a projective Hilbert $\G$-module. More precisely, we have a unitary isomorphism  
 \[
 C^1_{L^2}(G', \CC) \xrightarrow{\sim} \ell^2(\G) \otimes C^1(G, \CC) .
 \]
 To see this,  let $F$ be a fundamental domain for the action of $\G$ on $G'$. Then we have a decomposition of $G' \simeq \G \times F$ given by $gx \mapsto (g, x)$ for $g \in \G$, $x \in F$.
This induces a unitary isomorphism
 \[
 U_F \colon C^1_{L^2}(G', \CC) \xrightarrow{\sim} \ell^2(\G) \otimes C^1(F, \CC) \,  ,
 \]
 where the action of $L_g$ on  $C^1_{L^2}(G', \CC)$ becomes $L_g \otimes I$ on $\ell^2(\G) \otimes C^1(F, \CC)$. 
 We also have a unitary isomorphism 
 \[
 V_F \colon C^1(F, \CC) \xrightarrow{\sim} C^1(G, \CC) \, ,
 \]
 because $\phi$ restricts to an isomorphism from $F$ to $G$.
 
 It follows
 \[
 \dim_{\G}( C^1_{L^2}(G', \CC)) = \dim_{\CC}(C^1(G, \CC)) = |E(G)|\, .
 \]

 Similarly, one may consider $C^0_{L^2}(G', \CC)$, the $L^2$ space of functions on $V(G')$ (endowed with the counting measure). Then we also have a unitary isomorphism 
 \[
 C^0_{L^2}(G', \CC) \xrightarrow{\sim} \ell^2(\G) \otimes C^0(G, \CC) \, .
 \]
Consequently, 
 \[
 \dim_{\G}( C^0_{L^2}(G', \CC)) = \dim_{\CC}(C^0(G, \CC)) = |V(G)| \, .
 \]
 
 We note that our base-change from $\RR$ to $\CC$ is purely for convenience; the existing literature on von Neumann algebras is usually written over $\CC$.
\end{Example}

\section{Canonical measures on metric graphs}
\label{sec:canon}
\subsection{Canonical measures on compact metric graphs}
We will briefly review the notion of canonical measures on compact metric graphs. This notion was introduced in \cite{zhang1993admissible} (see also \cite{chinburg1993capacity}). In \cite{baker2011metric}, Baker and Faber give an interpretation in the framework of tropical geometry.

\begin{Definition}\label{def:zhang}
Let $\Gamma$ be a compact metric graph and $G$ be a model for $\Gamma$ with {\em no loops}. We let 
\begin{equation}\label{eq:eff_resis}
\Rsc(e)=
\begin{cases}
j^{\Gamma\backslash e}_{e^-}(e^+,e^+) & \text{, if } e \text{ is not a bridge,}\\
+\infty & \text{, if } e\text{ is a bridge.}
\end{cases}
\end{equation}
The canonical measure on $\Gamma$ is the piecewise Lebesgue measure determined by
\begin{equation}\label{eq:zhang}
\mu_{\can}|_e=\frac{1}{\Rsc(e)+\ell(e)} \, dx \, .
\end{equation}
\end{Definition}

The following result is well-known.
\begin{Proposition}\label{prop:proj}
	Let $\Gamma$ be a compact metric graph, $G$ be a model for $\Gamma$ without loops and with a fixed orientation $\Oc$, and let 
	$\pi \colon \Omega^1(\Gamma) \rightarrow \mathcal{H}(\Gamma)$
	be the orthogonal projection. Then for any $e \in \Oc$ we have 
	\[\mu_{\can}(e)=(\pi(de))_e = \frac{1}{\ell(e)} \langle \pi(de) , de \rangle\, .\]
\end{Proposition}

This follows from Kirchhoff's celebrated formula for the orthogonal projection matrix in terms of spanning trees \cite{Kirchhoff, Biggs}. See also \cite{baker2011metric}.

\subsection{Canonical measures on metric graphs}
\label{sec:CanGen}
Guided by Proposition~\ref{prop:proj}, we define the {\em canonical measure} on a (possibly non-compact) metric graph as follow:
\begin{Definition}\label{def:2}
	Let $\Gamma$ be a (possibly non-compact) metric graph. Let $G$ be a model with a fixed orientation $\Oc$, and let
	 \[\pi \colon \Omega^1_{L^2}(\Gamma) \rightarrow \mathcal{H}_{L^2}(\Gamma)\] 
	denote the orthogonal projection map. 
	The (generalized) canonical measure $\mu_{\can}$ (or $\mu^{\Gamma}_{\can}$ if we need to clarify the underlying graph) is the unique measure characterized by  
	\[\mu_{\can}(e) \coloneqq (\pi(de))_e = \frac{1}{\ell(e)} \langle \pi(de) , de \rangle \]
	for all $e \in \Oc$.
	
\end{Definition} 

It is easily seen that this measure does not depend on the choice of the orientation or the model $G$. Moreover, this measure is compatible with any model $G$, and is invariant under isometries of $\Gamma$.

An alternate characterization of canonical measures, which closely mirrors McMullen's definition of {\em Bergman metrics} for Riemann surfaces in \cite[(A.3)]{mcmullen2013entropy}, is the following:

\begin{Proposition} \label{prop:mcmul}
	Let $\Gamma$ be a metric graph with a model $G$ and an orientation $\Oc$. Let $e \in \Oc$. Let $\|\omega\| = \sqrt{\langle \omega , \omega \rangle}$ denotes the usual $L^2$ norm on forms.
	\begin{itemize} 
	\item[(a)] If $\mathcal{H}_{L^2}(\Gamma) \ne \{0\}$ we have
	\begin{equation}\label{eq:mc2}
	\ell(e)\mu_{can}(e)=\sup_{0 \ne \omega\in \mathcal{H}_{L^2}(\Gamma)} \frac{|\omega(e)|^2}{\|\omega\|^2}=\max_{0 \ne \omega\in \mathcal{H}_{L^2}(\Gamma)} \frac{|\omega(e)|^2}{\|\omega\|^2} \, .
\end{equation}
	\item[(b)] We always have
	\begin{equation}\label{eq:mc1}
	\ell(e)\mu_{can}(e)=\sup_{\omega\in \mathcal{H}_{L^2}(\Gamma), \|\omega\|\leq 1}|\omega(e)|^2=\max_{\omega\in \mathcal{H}_{L^2}(\Gamma), \|\omega\|\leq 1}|\omega(e)|^2 \, .
	\end{equation}
	\end{itemize}
	
\end{Proposition}

\begin{proof}
We will assume $\mathcal{H}_{L^2}(\Gamma) \ne \{0\}$; otherwise the statement in \eqref{eq:mc1} boils down to $0=0$ and there is nothing to prove.
We will first prove part (a).
Let $\pi \colon \Omega^1_{L^2}(\Gamma) \rightarrow \mathcal{H}_{L^2}(\Gamma)$ be the orthogonal projection as in Definition~\ref{def:2}. Then $de-\pi(de)$ is orthogonal to $\mathcal{H}_{L^2}(\Gamma)$ because $\pi^2 = {\rm id}$. For any $\omega\in \mathcal{H}_{L^2}(\Gamma)$ we have
\[\omega(e)=\langle\omega,de\rangle=\langle\omega,de-\pi(de)\rangle+\langle\omega,\pi(de)\rangle=\langle\omega,\pi(de)\rangle \, .
\]
So, by the Cauchy--Schwarz inequality, 
\begin{equation} \label{eq:cs}
 |\omega(e)| = |\langle\omega,\pi(de)\rangle | 
 \leq  \|\omega\| \|\pi(de)\| \, ,
\end{equation}
and the equality is achieved if and only if $\omega$ and $\pi(de)$ are linearly dependent. Therefore
\[
\sup_{0 \ne \omega\in \mathcal{H}_{L^2}(\Gamma)} \frac{|\omega(e)|^2}{\|\omega\|^2} = \|\pi(de)\|^2=\langle \pi(de),\pi(de)\rangle=\langle\pi(de),de\rangle=\ell(e)(\pi(de))_e=\ell(e)\mu_{\can}(e) \,.\]
This proves \eqref{eq:mc2}. To obtain \eqref{eq:mc1}, first note that, by \eqref{eq:cs}, we have
\[
\sup_{\omega\in \mathcal{H}_{L^2}(\Gamma), \|\omega\|\leq 1}|\omega(e)|^2 \leq \|\pi(de)\|^2 \, .
\]
But,  by \eqref{eq:mc2}, the equality is attained for a form of norm $1$.
\end{proof}

\begin{Remark} \label{rem:admissiblemeas}
Our results generalize immediately to ``augmented metric graphs'' (in the sense of \cite{ABBR1}), and  ``admissible measures'' (in the sense of \cite{zhang1993admissible}). There is an easy trick: if a point $p$ has genus $g(p)$, one can build a usual metric graph, where there are $g(p)$ loops (of arbitrary lengths) attached to the point $p$. Then, the total genus multiplied by the admissible measure is equal to the (measure-theoretic) pushforward of the canonical measure (in our sense) of the new metric graph under the map collapsing those loops. 
\end{Remark}

\subsection{Properties of the (generalized) canonical measure}
\subsubsection{Total mass}
We begin with the following well-known result, which is essentially due to Foster \cite{Foster} (see also \cite{Flanders, zhang1993admissible, baker2011metric}). \begin{Proposition}\label{prop:trace}
	Let $\Gamma$ be a compact graph, $G$ be a model for $\Gamma$. Then 
	\[
	\sum_{e\in E(G)} \mu_{\can}(e)=\dim_\RR H^1(\Gamma, \RR) \,.
	\]
	In other words, $\mu_{\can}(\Gamma)= g(\Gamma)$.

\end{Proposition}
The result can be thought of as a ``trace formula'', as both sides of the equality compute the trace of the projection matrix from $C^1(G,\RR)$ onto $H^1(G,\RR)$.

\subsubsection{Canonical measures and contractions}

Here we study the behavior of the canonical measure under contractions. 

Let $A \subseteq \Gamma$ be a subgraph. Then $\Gamma / A$ is the quotient metric graph whose equivalence classes are $A$ and all one point subsets $\{x\}$ for $x \not\in A$. Geometrically, one is contracting (collapsing) $A$ to a single point $p_A$.
The natural contraction map $c \colon \Gamma \rightarrow \Gamma/A$ is defined by $c(x) = x$ for $x \not\in A$, and $c(x) = p_A$ for $x \in A$.

We may choose a model $G$ with orientation $\Oc$ for $\Gamma$ such that $G$ contains a subgraph $H$ which is a model for $A$. Then $\Gamma / A$ has a model whose edge set is $E(G) \backslash E(H)$ and whose vertex set is $(V(G) \backslash V(H)) \cup \{p_A\}$. 
The contraction map $c \colon \Gamma\rightarrow\Gamma/A$ induces a map:
\[
c_* \colon  C^1_{L^2}(G,\RR)\rightarrow C^1_{L^2}(G/H,\RR)
\] 
define by $(c_*(\omega))_{e}=\omega_{e}$ for all $e$ not in $H$. 
Clearly $c_*$ contracts the $L^2$ norm. It is easy to check (using \eqref{adj}),
\[c_*(\mathcal{H}_{L^2}(\Gamma)) \subseteq \mathcal{H}_{L^2}(\Gamma/A) \,.\]

The contraction map $c \colon \Gamma\rightarrow\Gamma/A$ also induces a map:
\[
c^* \colon  C^1_{L^2}(G/H,\RR)\rightarrow C^1_{L^2}(G,\RR)
\] 
define by $(c^*(\omega))_{e}=\omega_{e}$ if $e$ is not in $H$, and $(c^*(\omega))_{e}=0$ if $e$ is in $H$. This map preserves the $L^2$ inner product, but it does not necessarily send harmonic forms to harmonic forms.

The following result is essentially equivalent to {\em Rayleigh's monotonicity law} in electrical networks (see \cite[Chapter 4]{DoyleSnell} or \cite[Chapter 2]{LP:book}, and \eqref{eq:zhang}):

\begin{Proposition}\label{lem3}
	Let $\Gamma$ be a metric graph, $e \subseteq \Gamma$ be any line segment. Let $A \subseteq \Gamma \backslash e$ be a compact subgraph. 
	Then
	\[\mu^{\Gamma}_{\can}(e) \leq \mu^{\Gamma / A}_{\can} (e) \, .\]
\end{Proposition}

\subsubsection{Approximation theorem}
The following result shows that the canonical measure on a non-compact graph can be approximated by canonical measures on any {\em exhaustion} by compact metric subgraphs. It follows immediately from \cite[Proposition 9.2]{LP:book}. See also \cite[\S7]{LyonsR}.

\begin{Proposition}\label{approx}
	Let $\Gamma$ be a non-compact metric graph, and  let $G$ be a model for $\Gamma$.
	\begin{itemize}
		\item[(i)] Let $\{A_i\}$ be an increasing sequence of compact metric subgraphs, compatible with $G$, such that $\Gamma = \bigcup A_i$. 
		\item[(ii)] Let $A'_i = \Gamma / \cl(\Gamma \backslash A_i)$.		
		\item[(iii)] Let $c_i \colon \Gamma \rightarrow A'_i$ be the contraction map.
		\item[(iv)] Let $\pi_i \colon \Omega^1(A'_i) \rightarrow \mathcal{H}(A'_i)$ and $\pi \colon \Omega^1_{L^2}(\Gamma) \rightarrow \mathcal{H}_{L^2}(\Gamma)$ be the orthogonal projections.
	\end{itemize}
	Then 
	\begin{itemize}
	\item[(a)] For any oriented edge $e$ of $G$, we have	
	\[
	\lim_{i\rightarrow\infty}c_i^*(\pi_i(de))=\pi(de) \,,
	\] 
	as elements in $C^1_{L^2}(G,\RR)$ with $L^2$ convergence. Note that, for any edge $e$, the projection $\pi_i(de)$ is well-defined for sufficiently large $i$.
	\item[(b)] 
 We have the convergence of measures (in the sense of \S\ref{sec:convmeas})
	\[
	\lim_{i\rightarrow\infty} {c_i^*( \mu_{\can}^{A'_i})} = \mu_{\can}^{\Gamma} 
	\] 
	on $\Gamma$. Here $c_i^*(\cdot)$ denotes pullback of measures (in the sense of \S\ref{sec:pullback}).
	\end{itemize}
	\end{Proposition}

\subsection{Induced measures from Galois covers}
\label{sec:}

Let $\phi \colon \Gamma'\rightarrow\Gamma$ be a Galois covering map between metric graphs. Since the canonical measure $\mu_{\can}^{\Gamma'}$ is invariant under isometries it is, in particular, invariant under the deck transformation group. Hence, there is a pushdown measure on $\Gamma$ of $\mu_{\can}^{\Gamma'}$ under the covering map $\phi $, which we will denote by $\mu_{\phi, \can} $.

\subsubsection{Finite Galois covers}
We first study $\mu_{\phi, \can} $ in the case that $\phi$ is a finite covering map between compact metric graphs:

\begin{Proposition} \label{prop:BergCovVol}
	Let $\phi \colon \Gamma' \rightarrow \Gamma$ be a Galois covering of degree $d$ between two compact metric graphs. 
Then
	\[
	\mu_{{\phi, \can}}(\Gamma)=g(\Gamma)-1+{1}/{d} \, .
	\]
\end{Proposition}

\begin{proof}
Let $G$ and $G'$ be models for $\Gamma$ and $\Gamma'$ compatible with $\phi$. We have
	\[
	\mu_{{\phi, \can}}(\Gamma) = \sum_{e\in E(G)}\mu_{{\phi, \can}}(e)={1\over d}\sum_{e\in E(G')} \mu^{\Gamma'}_{\can}(e)={1\over d}(1-\chi(\Gamma'))={1\over d}(1-d\chi(\Gamma))=g(\Gamma)-1+\frac{1}{d} \, .
	\] 
	The second equality follows from the definition of the pushdown measure. The third equality is Proposition~\ref{prop:trace}. The forth follows from
	$\chi(\Gamma')=d\chi(\Gamma)$.
\end{proof}

\subsubsection{Infinite Galois covers}
As one expects, there is an infinite covering analogue of Proposition~\ref{prop:BergCovVol}, which may be thought of as a {\em Gauss-Bonnet type theorem}. The proof is, however, much more subtle. 

\begin{Theorem}\label{thm:GaussBonnet}
Let $\phi \colon \Gamma' \rightarrow \Gamma$ be an infinite Galois covering of a compact metric graph $\Gamma$. Then
	\begin{equation} \label{eq:trformula}
	\mu_{\phi, \can} (\Gamma)=g(\Gamma)-1 \,  .
	\end{equation}
\end{Theorem}
\begin{proof} 
Let $G$ and $G'$ be compatible models for $\Gamma$ and $\Gamma'$. Let $H=\pi_1(\Gamma)/\pi_1(\Gamma')$ be the deck transformation group. 

Let $\pi \colon C^1_{L^2}(G',\RR)\rightarrow\mathcal{H}_{L^2}(\Gamma')$ be the orthogonal projection. Let $C^0_{L^2}(G', \CC)$ be the $L^2$ space of functions on $V(G')$ (endowed with the counting measure). We first prove the sequence
\begin{equation} \label{eq:weakexact}
0 \rightarrow C^0_{L^2}(G',\RR) \xrightarrow{d} C^1_{L^2}(G',\RR) \xrightarrow{\pi} \mathcal{H}_{L^2}(\Gamma') \rightarrow 0
\end{equation}
is weakly exact.
To see this, consider the coboundary map $d \colon  C^0_{L^2}(G',\RR)\rightarrow C^1_{L^2}(G',\RR)$.
\begin{itemize} [noitemsep, leftmargin=*]
\item[-] $d$ is continuous: it is a bounded operator whose operator norm is bounded above by $
 \sup\{{2\delta}/{\sqrt{\ell(e)}} \colon {e\in E(G')}\}$, 
 where $\delta$ denotes the maximum valence of vertices in $G$. This is a finite number, by our definition of metric graphs. 
 \item[-] $d$ is injective: if $f\in C^0_{L^2}(G',\RR)$ has $df=0$ then $f$ is constant. Since it is also $L^2$ integrable, we must have $f=0$. 
 \item[-] ${\rm Kernel}(\pi) = \cl({\rm Image}(d))$: by Proposition~\ref{prop:df w orth}~(c) we know ${\rm Kernel}(\pi) = \cl\left(d(C^0_c(G,\RR))\right)$. But $C^0_c(G',\RR)$ is dense in $C^0_{L^2}(G',\RR)$ and $d$ is continuous. Therefore $\cl\left(d(C^0_c(G,\RR))\right)=\cl(d(C^0_{L^2}(G',\RR))$.
 \end{itemize} 

Tensoring \eqref{eq:weakexact} with $\CC$, we conclude that the sequence:
\[
0 \rightarrow C^0_{L^2}(G',\CC) \xrightarrow{d} C^1_{L^2}(G',\CC) \xrightarrow{\pi} \mathcal{H}_{L^2}(\Gamma') \otimes \CC \rightarrow 0
\]
is also weakly exact and hence, by Remark~\ref{rmk:dimGprops}~(iv), we have
\[
\dim_{\G}(C^1_{L^2}(G',\CC)) = \dim_{\G}(\mathcal{H}_{L^2}(\Gamma') \otimes \CC) + \dim_{\G}(C^0_{L^2}(G',\CC)) \, .
\]
By Example \ref{eg:C1dim} we obtain:
\begin{equation}\label{eq:trace1}
 \dim_{\G}(\mathcal{H}_{L^2}(\Gamma') \otimes \CC) = |E(G)| - |V(G)| = g(\Gamma) - 1\,  .
 \end{equation}

If $\pi \colon C^1_{L^2}(G',\RR) \rightarrow \mathcal{H}_{L^2}(\Gamma')$ is the orthogonal projection, let $\pi' \colon C^1_{L^2}(G',\CC) \rightarrow \mathcal{H}_{L^2}(\Gamma') \otimes \CC$ be the projection $\pi \otimes 1$ (extension of scalars), and let $\langle \cdot , \cdot \rangle_{\CC} \colon  C^1_{L^2}(G',\CC) \times  C^1_{L^2}(G',\CC) \rightarrow \CC$ be the extension of $\langle \cdot , \cdot \rangle \colon  C^1_{L^2}(G',\RR) \times  C^1_{L^2}(G',\RR) \rightarrow \RR$. 

An orthonormal basis for $C^1_{L^2}(G',\CC) \simeq \ell^2(\G) \otimes (C^1(G, \RR) \otimes \CC)$ is 
\[
\left\{ \delta_g \otimes ( de \otimes 1)/{\sqrt{\ell(e)}} \colon g\in\G, e\in \mathcal{O} \right\} \, .
\] 
For $x_e =  \delta_{{\rm id}} \otimes (de \otimes 1) /{\sqrt{\ell(e)}} \in \ell^2(\G) \otimes (C^1(G, \RR) \otimes \CC)$, it follows from Definition~\ref{def:2} that 
\begin{equation}\label{eq:trace2}
\langle \pi'(x_e) , x_e \rangle_{\CC} = \frac{1}{\ell(e)}\langle \pi(\delta_{{\rm id}} \otimes de), \delta_{{\rm id}} \otimes de \rangle = \mu_{\can}^{\Gamma'}(e) \, .
\end{equation}
On the other hand, by Definition~\ref{def:trG} and Definition \ref{def:dimG}, we know
\begin{equation}\label{eq:trace3}
 \dim_{\G}(\mathcal{H}_{L^2}(\Gamma') \otimes \CC) = \sum_{e \in \Oc} \langle \pi'(x_e) , x_e \rangle_{\CC} \, .
\end{equation}

The result follows from putting together \eqref{eq:trace1}, \eqref{eq:trace2}, and \eqref{eq:trace3}.
\end{proof}
\begin{Remark} \phantomsection \label{rmk:GB}  
\begin{itemize}
\item[]
\item[(i)] As is clear from the proof, one might interpret \eqref{eq:trformula} in Theorem~\ref{thm:GaussBonnet} as a ``trace formula''.
\item[(ii)] It is well known that $g(\Gamma)-1$ (the right-hand side of \eqref{eq:trformula}) is the first (and the only nonzero) $L^2$ Betti number of $\Gamma$. So one might interpret the induced measure $\mu_{\phi, \can}$ on $\Gamma$ as giving a {\em measure-theoretic decomposition} of the first $L^2$ Betti number. This measure theoretic decomposition is independent of the choice of models for $\Gamma$.
\item[(iii)] Our base-change from $\RR$ to $\CC$ in the proof of Theorem~\ref{thm:GaussBonnet} is for convenience; the existing literature on von Neumann algebras is usually written over $\CC$.
\end{itemize}
\end{Remark}
\begin{Remark} \phantomsection
We are also able to prove Theorem~\ref{thm:GaussBonnet} {\em without} the use of von Neumann algebras and dimensions. However, that proof is far more technical, especially in the case that the infinite cover is not transient (in the sense of Definition~\ref{def:transient}). Moreover, the ``trace formula'' interpretation of \eqref{eq:trformula} is not clear from the more technical proof. 
\end{Remark}

\section{A generalized Kazhdan's theorem for metric graphs}
\label{sec:kazh}
In this section, we will show that the canonical measure satisfies a (generalized) Kazhdan-type theorem. 

\begin{Theorem} \label{thm:Kazhdan}
	Let $\phi \colon \Gamma' \rightarrow \Gamma$ be an infinite Galois covering of compact metric graph $\Gamma$. Let $\{\phi_n \colon \Gamma_n \rightarrow \Gamma \colon n \geq 1 \}$ be an ascending sequence of finite Galois covers converging to $\Gamma'$, in the sense that the equality
	\begin{equation}\label{eq:pi1intersection}
	\bigcap_{n \geq 1}  \pi_1(\Gamma_n) = \pi_1(\Gamma')
	\end{equation}
	holds in $\pi_1(\Gamma)$. 
Then we have the strong convergence of measures $\lim_{n\rightarrow\infty}\mu_{{\phi_n, \can}}=\mu_{\phi, \can}$.
	\end{Theorem}
\begin{Remark} 
\begin{itemize}
\item[]
		\item[(i)] We have omitted the base points for the fundamental groups in \eqref{eq:pi1intersection}, because by our Galois assumption, we have $\pi_1(\Gamma') \unlhd \pi_1(\Gamma)$ and $\pi_1(\Gamma_n) \unlhd \pi_1(\Gamma)$ for all $n$.
		\item[(ii)] One can replace the tower of covers $\{\Gamma_n\}$ in Theorem~\ref{thm:Kazhdan} with a {\em cofinal} family $\mathcal{F}$ of finite covers between $\Gamma$ and $\Gamma'$, provided that the covering group $\pi_1(\Gamma) / \pi_1(\Gamma')$ is {\em residually finite}. In particular, this applied to the case where $\Gamma'$ is the universal cover of $\Gamma$, in which case $\pi_1(\Gamma) / \pi_1(\Gamma') = \pi_1(\Gamma) \simeq F_g$ is a free group on $g$ generators. Note, however, it is not hard to find examples of quotients of free groups which are not residually finite (see e.g. \cite[Section 2]{magnus}).
	\end{itemize} 
	\end{Remark}
	
	\begin{proof} 
	
		Let $G$, $G_n$, and $G'$ denote models for $\Gamma$, $\Gamma_n$, and $\Gamma'$, compatible with the covering maps.

\noindent {\bf Claim 1.} For any $e\in E(G)$,
\begin{equation}\label{sup}
		\limsup_{n}\mu_{{\phi_n, \can}}(e)\leq \mu_{{\phi, \can}}(e) \, .
		\end{equation}
		\noindent {\em Proof of Claim 1.} Let $e' \in E(G')$ be a fixed lift of $e \in E(G)$. Let $\{A_i\}$ be an increasing sequence of compact subgraphs compatible with $G'$ and containing $e'$, such that $\Gamma' = \bigcup_i A_i$. Let $A'_i=\Gamma'/\cl(\Gamma'\backslash A_i)$. By Proposition~\ref{approx}~(b), we have convergence 		
		\[
		\lim_i{\mu_{\can}^{A'_i}}(e') = \mu_{\can}^{\Gamma'}(e') = \mu_{{\phi, \can}}(e) \, .
		\]

           Because $A_i$ has a model which is a finite subgraph of $G'$, it contains only finitely many vertices of $G'$. Hence, the number of simple paths in $A_i$ starting and ending at vertices of $G'$ must be finite. Consider the simple paths in $A_i$ that are sent to loops $\gamma$ in $\Gamma$. Each such loop $\gamma$ represents a conjugacy class $[\gamma]$ in $\pi_1(\Gamma)$ which is not contained in $\pi_1(\Gamma')$. Because $\bigcap \pi_1(\Gamma_n)=\pi_1(\Gamma')$, for each conjugacy class $[\gamma]$ there is some integer $N_\gamma$ such that $[\gamma]\not\subseteq\pi_1(\Gamma_{N_\gamma})$. Let $N_i$ be the maximum of all these finitely many $N_\gamma$. Then $A_i$ is sent isometrically to its image under $\phi_{N_i}$ as no two points in $A_i$ can be sent to the same point in $\Gamma_{N_i}$. Because all our coverings are Galois, the deck transformation group acts transitively on preimages of any edge. Hence, for any lift $e''$ of $e$ in any $\Gamma_n$, $n\geq N_i$, there is a subgraph in $\Gamma_n$ isometric to $A_i$ such that the isometry sends $e''$ to $e'$. Therefore, by Proposition~\ref{lem3}, for $n\geq N_i$ we have
		\[
		\mu_{{\phi_n, \can}}(e) =  \mu_{\can}^{\Gamma_n}(e'') \leq \mu_{\can}^{A'_i}(e') \,  .
		\]
Therefore
                \[
		\limsup_{n}\mu_{{\phi_n, \can}}(e) \leq \limsup_i \mu_{\can}^{A'_i}(e') = \lim_i \mu_{\can}^{A'_i}(e') = \mu_{{\phi, \can}}(e) \, .
		\]
                
\noindent {\bf Claim 2.} For any $e\in E(G)$,
\begin{equation}\label{sup2}
		 \mu_{{\phi, \can}}(e) \leq \liminf_{n}\mu_{{\phi_n, \can}}(e) \, .
		\end{equation}
		\noindent {\em Proof of Claim 2.} Consider the function $F_n \colon E(G) \rightarrow \RR$ defined by
		\[F_n(f)=\mu_{{\phi_n, \can}}(f)-\mu_{{\phi, \can}}(f) \, .\]
		
		By Claim 1 we know, for all $f \in E(G)$, 
		\begin{equation} \label{eq:Fneg}
		\limsup_n F_n (f) \leq 0 \, .
		\end{equation}
		On the other hand, using Proposition~\ref{prop:BergCovVol} and Theorem~\ref{thm:GaussBonnet}, we have
		\begin{equation}\label{eq:sumF}
		\lim_n \left(\sum_{f \in E(G)} F_n(f)\right)=\lim_n \left(\mu_{{\phi_n, \can}}(\Gamma)-\mu_{\phi, \can}(\Gamma)\right)= 0 \, .
		\end{equation}
Claim 2 follows from the following: 
		\[
		\begin{aligned}
		\liminf_n F_n(e) &= \liminf_n \left(\sum_{f}F_n(f) - \sum_{f \ne e}F_n(f)\right)\\
		 &\geq \liminf_n \left(\sum_{f}F_n(f) \right) +\liminf_n\left(- \sum_{f \ne e}F_n(f)\right) \\
		 &= 0 - \limsup_n\left(\sum_{f \ne e}F_n(f)\right)	&\text{by \eqref{eq:sumF}}\\	 
		 &\geq -\sum_{f \ne e}\left(\limsup_n F_n(f)\right)\\
		 &\geq 0 &\text{by \eqref{eq:Fneg}} \, .
		\end{aligned}
		\] 
		Since $\Gamma$ is compact, it follows from \eqref{sup} and \eqref{sup2} that we have the desired strong convergence of measures.
			\end{proof}

\section{Interpretations of limiting measures}
\label{sec:examples} 
In this section, we provide a few examples for measures inherited from infinite covers. We will study canonical measures on universal covers in detail, as one might want to consider them as an analogue of the ``hyperbolic measure''. We will study this case from a computational point of view. We will also relate the ``hyperbolic measure'' to other notions such as Poisson-Jensen and equilibrium measures.

\subsection{Basic examples}
\label{sec:ex}

We start with the following easy observation:
\begin{Lemma}\label{uniform}
	Let $\Gamma$ be a metric graph and let $G$ a model for $\Gamma$. Assume the group of isometries of $\Gamma$ acts transitively on $E(G)$. Then the canonical measure $\mu_{\can}$ on $\Gamma$ is proportional to the standard Lebesgue measure on $\Gamma$. 
\end{Lemma}
\begin{proof}
The canonical measure is preserved under isometries, so the value of $\mu_{\can}$ evaluated on every edge must be the same.
\end{proof}
As an immediate corollary, we have the following result on regular graphs. Recall a (combinatorial) graph is called $k$-regular if every vertex has the same valency (or degree) $k$. 

\begin{Corollary} \label{cor:reg}
	Let $\Gamma$ be a compact metric graph, $G$ be a model for $\Gamma$ such that $G$ is a $k$-regular graph where $k\geq 3$. Assume all edges of $G$ have the same length. Let $\phi \colon \Gamma' \rightarrow \Gamma$ be the universal covering map. Then $\mu_{\phi, \can}(e)={1-2/k}$ for all $e\in E(G)$.
\end{Corollary}	
\begin{proof}
Since $\Gamma'$ is an infinite $k$-regular tree with uniform edge lengths, we are in the situation of Lemma~\ref{uniform}, and the canonical measure on $\Gamma'$ is proportional to the Lebesgue measure.  
	From Theorem~\ref{thm:GaussBonnet}, we know the pushdown of this canonical measure under $ \phi \colon\Gamma'\rightarrow \Gamma$ has total mass $g(\Gamma) - 1$. Let $|E(G)|=m$ and $|V(G)|=n$.  We have $m \mu_{\phi, \can}(e) =  g(\Gamma)-1 = m-n$, and therefore	$\mu_{\phi, \can}(e) = 1-{n}/{m} = 1-{2}/{k}$. The last equality is by the handshake lemma: $kn=2m$.
\end{proof}

Next example shows that, on the same metric graph, different infinite covers might induce different limiting measures.
\begin{Example}\label{ex2} 
	Let $\Gamma$ be a metric graph with a model $G$, which is a ``rose'' consisting of one vertex and $d$ loops of length $1$ (see Figure~\ref{fig:rose}).
	
	\begin{itemize}
	\item[(i)] Let $\phi' \colon \Gamma' \rightarrow \Gamma$ be the maximal abelian covering. Then $G$ is the usual $\mathbb{Z}^d$ grid graph. By Lemma~\ref{uniform}, we know that $\mu_{\phi', \can}(\cdot)$ is constant on $E(G)$. By Theorem~\ref{thm:GaussBonnet} we know $\mu_{\phi', \can}(\Gamma) = d-1$. Therefore $\mu_{\phi', \can}(e)=1-{1/d}$ for all $e\in E(G)$.

 \begin{figure}[H]
   \begin{tikzpicture}
     \draw[-](-0.5,0)circle(0.5);
     \draw[-](0.5,0)circle(0.5);
     \node at (-1.3,0){$e_1$};
     \node at (1.3,0){$e_2$};
     \draw[fill](0,0) circle (2pt);

     \draw[-](3,0)--(7,0);
     \draw[-](3,1)--(7,1);
     \draw[-](3,-1)--(7,-1);
     \draw[-](4,-2)--(4,2);
     \draw[-](5,-2)--(5,2);
     \draw[-](6,-2)--(6,2);
     
     \draw[fill](4,0) circle (2pt);
  \draw[fill](5,0) circle (2pt);
    \draw[fill](6,0) circle (2pt);
      \draw[fill](4,1) circle (2pt);
        \draw[fill](5,1) circle (2pt);
          \draw[fill](6,1) circle (2pt);
            \draw[fill](4,-1) circle (2pt);
              \draw[fill](5,-1) circle (2pt);
                \draw[fill](6,-1) circle (2pt);
     \node at (3.7,0.5){$e'_1$};

     \node at (4.5,1.3){$e'_2$};

\end{tikzpicture}   
\caption{\label{fig:rose}
  Left: a rose graph with $2$ loops. 
  Right: the maximal abelian covering, the $\ZZ^2$ grid graph.}
  \end{figure}
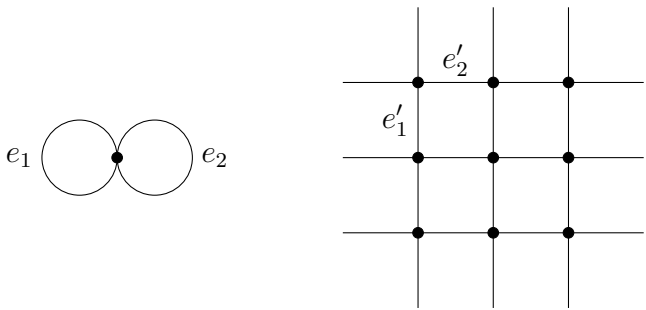
\vspace{-1mm}
\item[(ii)]
	Let $d=4$ and let $E(G)=\{e_1, e_2, e_3, e_4\}$. Consider the surjective homomorphism $\pi_1(\Gamma) \twoheadrightarrow \mathbb{Z}^3$ which sends $e_1,e_2,e_3$ to the three standard generators of $\mathbb{Z}^3$, while sending $e_4$ to $0$. 
The Galois covering space $\phi'' \colon \Gamma'' \rightarrow \Gamma$ corresponding to this homomorphism is the usual $\mathbb{Z}^3$ grid with one loop attached to each vertex. It follows from part (i) and Lemma~\ref{lemma:bridge} that the induced measure on $\Gamma$ is given by $\mu_{\phi'', \can}(e_1)=\mu_{\phi'', \can}(e_2)=\mu_{\phi'', \can}(e_3)={2/3}$, $\mu_{\phi'', \can}(e_4)=1$. 
\end{itemize}

\end{Example}

\subsection{Transient metric graphs and universal covers} 

\subsubsection{Transient metric graphs}
Infinite covers which are ``transient'' are easier to work with. The following result is essentially \cite[Theorem 2.10]{LP:book}.
\begin{Lemma}\label{lem:anyall}
Let $\Gamma$ be a metric graph. The following are equivalent:
\begin{itemize}
\item[(i)] For {\em some} $x \in \Gamma$, there exists $\omega\in\Omega^1_{L^2}(\Gamma)$ such that $d^* \omega=\delta_x$.
\item[(ii)] For {\em all} $x \in \Gamma$, there exists $\omega\in\Omega^1_{L^2}(\Gamma)$ such that $d^* \omega=\delta_x$.
\end{itemize}
\end{Lemma}

\begin{Definition}\label{def:transient}
A metric graph $\Gamma$ is called {\em transient}, if it satisfies the equivalent conditions in Lemma~\ref{lem:anyall}. 
\end{Definition}
\begin{Remark}
The name ``transient'' is compatible with the notion of transience of random walks on infinite combinatorial graphs: a metric graph $\Gamma$ is transient in our sense if and only if any model $G$ of $\Gamma$ is transient in the sense of random walks. See \cite[Chapter 2]{LP:book} and references therein for the random walk literature.
\end{Remark}
The following result implies that the universal cover of a metric graph of genus at least $2$ is always transient. 
\begin{Proposition} \label{prop:transient}
Let $\Gamma$ be a metric graph. Assume 
\begin{itemize}
\item[(i)] there exists a model $G$ for $\Gamma$ which is an infinite tree with only finitely many valence $1$ and $2$ vertices, 
\item[(ii)] there exists a real number $C$ such that for all $e \in E(G)$, we have $\ell(e) <C$, i.e. the edge lengths are uniformly bounded from above.
\end{itemize}
Then $\Gamma$ is transient.
\end{Proposition}
 
We will need the following terminology for the proof. A {\em full} binary tree is a rooted binary tree in which every node has two children. The {\em level} of a vertex in a rooted tree is the number of edges for the shortest walk from the root.
 
\begin{proof}
  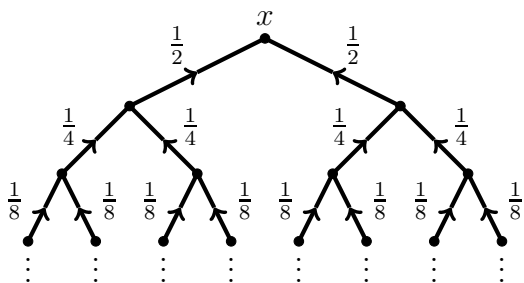
\begin{figure}[H]
   \begin{tikzpicture}[scale=.9]
        \node at (0,0.3){$x$};
     \draw[black, ultra thick, ->](-2,-1)--(-1,-.5);
     \draw[black, ultra thick, -](-1,-.5)--(0,0);
     \draw[black, ultra thick, ->](2,-1)--(1,-.5);
     \draw[black, ultra thick, -](1,-.5)--(0,0);
     
     \draw[black, ultra thick, ->](-3,-2)--(-2.5,-1.5);
     \draw[black, ultra thick, -](-2.5,-1.5)--(-2,-1);
     \draw[black, ultra thick, ->](-1,-2)--(-1.5,-1.5);
     \draw[black, ultra thick, -](-1.5,-1.5)--(-2,-1);

     \draw[black, ultra thick, ->](1,-2)--(1.5,-1.5);
     \draw[black, ultra thick, -](1.5,-1.5)--(2,-1);
     \draw[black, ultra thick, ->](3,-2)--(2.5,-1.5);
     \draw[black, ultra thick, -](2.5,-1.5)--(2,-1);
     
          \draw[black, ultra thick, ->](-3.5,-3)--(-3.25,-2.5);
     \draw[black, ultra thick, -](-3.25,-2.5)--(-3,-2);
     \draw[black, ultra thick, ->](-2.5,-3)--(-2.75,-2.5);
     \draw[black, ultra thick, -](-2.75,-2.5)--(-3,-2);
     
          \draw[black, ultra thick, ->](-1.5,-3)--(-1.25,-2.5);
     \draw[black, ultra thick, -](-1.25,-2.5)--(-1,-2);
     \draw[black, ultra thick, ->](-.5,-3)--(-.75,-2.5);
     \draw[black, ultra thick, -](-.75,-2.5)--(-1,-2);
     
               \draw[black, ultra thick, ->](1.5,-3)--(1.25,-2.5);
     \draw[black, ultra thick, -](1.25,-2.5)--(1,-2);
     \draw[black, ultra thick, ->](.5,-3)--(.75,-2.5);
     \draw[black, ultra thick, -](.75,-2.5)--(1,-2);
     
          \draw[black, ultra thick, ->](3.5,-3)--(3.25,-2.5);
     \draw[black, ultra thick, -](3.25,-2.5)--(3,-2);
     \draw[black, ultra thick, ->](2.5,-3)--(2.75,-2.5);
     \draw[black, ultra thick, -](2.75,-2.5)--(3,-2);
     
       \draw[fill](0,0) circle (2pt);
       \draw[fill](2,-1) circle (2pt);
       \draw[fill](-2,-1) circle (2pt);
       \draw[fill](-3,-2) circle (2pt);
       \draw[fill](-1,-2) circle (2pt);
        \draw[fill](1,-2) circle (2pt);
      \draw[fill](3,-2) circle (2pt);
              \draw[fill](-3.5,-3) circle (2pt);
       \draw[fill](-2.5,-3) circle (2pt);
              \draw[fill](-1.5,-3) circle (2pt);
       \draw[fill](-.5,-3) circle (2pt);
              \draw[fill](.5,-3) circle (2pt);
       \draw[fill](1.5,-3) circle (2pt);
              \draw[fill](2.5,-3) circle (2pt);
       \draw[fill](3.5,-3) circle (2pt);             

     \node at (-1.3,-.1){$1\over 2$};
     \node at (1.3,-.1){$1\over 2$};
     \node at (-2.9,-1.3){$1\over 4$};
     \node at (-1.1,-1.3){$1\over 4$};
     \node at (1.1,-1.3){$1\over 4$};
     \node at (2.9,-1.3){$1\over 4$};
     \node at (-3.7,-2.4){$1\over 8$};
     \node at (-2.3,-2.4){$1\over 8$};
     \node at (-0.3,-2.4){$1\over8$};
     \node at (-1.7,-2.4){$1\over 8$};
     \node at (3.7,-2.4){$1\over 8$};
     \node at (2.3,-2.4){$1\over 8$};
     \node at (0.3,-2.4){$1\over8$};
     \node at (1.7,-2.4){$1\over 8$};
     
     \node at (-3.5,-3.3){$\vdots$};
          \node at (-2.5,-3.3){$\vdots$};
               \node at (-1.5,-3.3){$\vdots$};
                    \node at (-.5,-3.3){$\vdots$};
                         \node at (3.5,-3.3){$\vdots$};
                              \node at (2.5,-3.3){$\vdots$};
                                   \node at (1.5,-3.3){$\vdots$};    
                                    \node at (.5,-3.3){$\vdots$};

\end{tikzpicture} 
\caption{\label{fig:full}The full infinite binary subtree of $\Gamma$ rooted at $x$. The labels on edges denote the values of $\omega'_e$.}        
\end{figure}

It follows from the assumptions that $\Gamma$ contains a subgraph $\Gamma'$ with a model $G'$ which is a full, infinite, binary tree rooted at some vertex $x$. Clearly $C$ is an upper bound for the edge lengths in $G'$ as well. We fix an orientation on $G'$ such that all edges are directed towards the root $x$. We define $\omega' \in C^1(G' , \RR)$ by $\omega'_e = 2^{-{\rm level}(e^-)}$ (see Figure~\ref{fig:full}).
Consider its class in $\Omega^1(\Gamma')$, denoted again by $\omega'$.  The extension by zero of $\omega'$ gives a form $\omega \in \Omega^1(\Gamma)$ which has the property $d^*\omega=\delta_x$. Therefore the condition in Lemma \ref{lem:anyall}~(i) is satisfied for $\Gamma$.
\end{proof} 

\begin{Remark}
One can alternatively prove Proposition~\ref{prop:transient} using a theorem of Lyons \cite{Lyons} (see also \cite[Chapter 2]{LP:book}).
\end{Remark}

\subsubsection{Universal covers and the corresponding limiting measures}
\label{sec:ucover1}
Let $\Gamma$ be a compact metric graph with $g(\Gamma) \geq 2$, and let $\phi \colon \Gamma'\rightarrow\Gamma$ be the universal cover. 

  Let $G'$ be a model for $\Gamma'$ compatible with a model $G$ for $\Gamma$, so $V(G')$ is the preimage of $V(G)$ under the covering map. Because $\Gamma'$ is a universal cover, $G'$ is a tree, hence every edge of $G'$ must be a bridge. Given any directed edge $e'$ in $G'$ which is a lift of a directed edge $e$ in $G$, let $T(e')$ be the connected component of $G'$ minus the interior of $e'$, which contains ${e'}^+$ (see Figure~\ref{fig:univT}).

  \begin{figure}[H]
\begin{tikzpicture}[scale=.75]
\draw[black, ultra thick, -](.5,0)--(1,0);
\draw[black, ultra thick, ->](0,0)--(.5,0);
\draw[dotted](-1,0) circle (1);
\node at (.5,0.5){$e'$};
\draw[dotted](2,0) circle (1);
\node at (-1,0){$T(\overline{e}')$};
\node at (+2,0){$T(e')$};
\end{tikzpicture}
\caption{\label{fig:univT} The subgraph $T(e')$ corresponding to $e'$ in the universal cover.}
\end{figure}
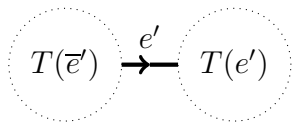 

  By the definition of covering map, different choices of the lift of $e$ result in isometric metric graphs $T(e')$. By Proposition~\ref{prop:transient}, we know $T(e')$ is not transient if and only if it is compact.  
  \begin{Definition}\label{def:R}
  Let $R \colon \EE(G) \rightarrow \RR^{>0} \cup \{+\infty \}$ be the function defined as follows:
  \begin{equation}\label{eq:Runiv}
  R(e) =
  \begin{cases}
+\infty &\text{, if $T(e')$ is not transient,}\\
\inf{\| \omega \|^2} &\text{, if $T(e')$ is transient.}\\
\end{cases}  
  \end{equation}
  Where the infimum is taken over all $\omega\in C^1_{L^2}(T(e'),\RR) \subset C^1_{L^2}(G',\RR)$ with $d^*\omega=\delta_{{e'}^+}$.

       \end{Definition}
\begin{Remark} \phantomsection \label{rmk:Rprop}
\begin{itemize}
\item[]
\item[(i)] The value of $R(e)$ is well-defined because different choice of the lift $e'$ result in isometric $T(e')$, the value of $R(e)$ is independent of the choice of the lift $e'$.
\item[(ii)] When $T(e')$ is transient, the infimum in \eqref{eq:Runiv} is always attained by a unique $\nu_{e'} \in C^1_{L^2}(G',\RR)$, so $R(e) = \|\nu_{e'}\|^2$. This is because $d$ is a bounded operator (see the proof of Theorem~\ref{thm:GaussBonnet}) hence so is $d^*$, and $d^*\omega=\delta_{{e'}^+}$ defines an affine closed subspace of $C^1_{L^2}(T(e'),\RR)$. By the Hilbert projection theorem, on any affine closed subspace of a Hilbert space, there is a unique point for which the norm is minimized.
\item[(iii)] The value $R(e)$ has the following electrical network interpretation:     in the case that $T(e')$ is not compact, identify all of its ends with a single point at infinity $p_\infty$. Then $R(e)$ is the effective resistance between $e'^+$ and $p_\infty$. This intuition comes from ``Thompson's principle'' which states the electrical current is the unique flow that minimizes the energy functional. Using a similar interpretation, one could extend some results on universal covers (e.g. Theorem~\ref{cal1}) to all {\em transient} covers.
\item[(iv)] In Theorem~\ref{cal2}, we will give {\em algebraic equations}, in terms of the finite graph $G$, that one can use to compute the values $R(e)$.
\end{itemize}
\end{Remark}

Our next result gives a formula for the canonical measure on the universal cover $\Gamma'$, and its pushdown on $\Gamma$. The formula \eqref{conc1} should be compared with \eqref{eq:zhang}.

\begin{Theorem}\label{cal1}
	Let $\Gamma$ be a compact metric graph with $g(\Gamma) \geq 2$. Let $\phi \colon \Gamma'\rightarrow\Gamma$ be the universal covering map. Let $\mu_{{\phi, \can}}$ be the pushdown of the canonical measure of $\Gamma'$, and let $G$ be a model for $\Gamma$. For $e \in \EE(G)$ we have 
\begin{equation}\label{conc1}
\mu_{\phi, \can}|_e=\frac{1}{\mathscr{S}(e)+\ell(e)} \, dx \, ,
\end{equation}
where $\mathscr{S}(e) = R(e)+R(\overline{e})$.
\end{Theorem}
\begin{proof}	
If either $T(e')$ or $T(\overline{e}')$ is not transient then $\mathscr{S}(e)=+\infty$. But in this case $\mu_{\phi, \can}(e)=0$ by Lemma~\ref{lemma:bridge}.
So we may assume that both $T(e')$ or $T(\overline{e}')$ are transient. Let $\nu_{e'}$ and $\nu_{\overline{e}'}$ be as in Remark~\ref{rmk:Rprop}~(ii), so $R(e) = \|\nu_{e'}\|^2$ and $R(\overline{e}) = \|\nu_{\overline{e}'}\|^2$.
 
Let $\omega\in \mathcal{H}_{L^2}(\Gamma')$ with $\|\omega\| = 1$ be the form attaining the maximum in \eqref{eq:mc2} (Proposition~\ref{prop:mcmul}), so $\ell(e)\mu_{can}^{\Gamma'}(e')={|\omega(e')|^2}$. By elementary convex duality, this $\omega$ has the smallest possible norm among all harmonic forms $\alpha \in \mathcal{H}_{L^2}(\Gamma')$ with $|\alpha(e)| = |\omega(e)|$:
\[
\|{\|\omega\|\over \|\alpha\|}\alpha\|=\|\omega\| \  \implies \  |{\|\omega\|\over \|\alpha\|}\alpha(e')|\leq |\omega(e')| \  \implies \  {\|\omega\|\over \|\alpha\|}\leq {|\omega(e')|\over|\alpha(e')|}=1 \, .
\]

Let $\beta = de'-\nu_{e'}+\nu_{\overline{e}'}$. We have:
\begin{itemize}
\item $\beta$ is harmonic on $\Gamma'$ because
$d^*( de'-\nu_{e'}+\nu_{\overline{e}'} ) = (\delta_{e'^+} - \delta_{e'^-}) - (\delta_{e'^+}) + (\delta_{\overline{e}'^+}) = 0$.
\item $\beta$ has the smallest norm between all harmonic forms $\alpha \in \mathcal{H}_{L^2}(\Gamma')$ with $\alpha(e) = \beta(e') = \ell(e')$. This is because $\|\beta\|^2 = \|de'\|^2 + \| \nu_{e'}\|^2 +\|\nu_{\overline{e}'}\|^2$.

\end{itemize}

By the uniqueness of the minimizers, we conclude 
\[
\omega = \frac{|\omega(e')|}{|\beta(e')|} \, \beta 
\]
and, therefore
\[
1 = \|\omega\|^2 = \frac{|\omega(e')|^2}{|\beta(e')|^2} \|\beta\|^2 = \frac{\ell(e)\mu_{can}^{\Gamma'}(e')}{\ell(e)^2} (\ell(e) + \mathscr{S}(e)) \, .
\]
We conclude
$\mu_{can}^{\Gamma'}(e') = {\ell(e)}/{(\ell(e) + \mathscr{S}(e))}$,
which completes the proof.
\end{proof}

\begin{Theorem}\label{cal2}
Let $G$ be a finite metric graph. Then the function $R$ in Definition~\ref{def:R} satisfies: 
	\begin{equation}\label{conc2}
	{1\over R(e)}=\sum_{e_i \in S_e}{1\over \ell(e_i)+R(e_i)}
	\end{equation}
	for all $e \in \EE(G)$, where $S_e = \{e_i \in \EE(G) \colon e_i \ne \overline{e} , e_i^-=e^+\}$.  

\end{Theorem}
\begin{Remark}
\begin{itemize}
\item[]
\item[(i)]  One could alternatively prove Theorem~\ref{cal1} and Theorem~\ref{cal2} using \cite[Theorem 4.9]{aminiFactorials}. One could also deduce from \cite[Theorem 4.9]{aminiFactorials} that (35) has a positive solution. 
\item[(ii)]  We expect \eqref{conc2} will always have a {\em unique} positive solution. The intuition is that \eqref{conc2} may be thought of as ``parallel laws'' in the infinite electrical network given by the universal cover. Since the universal cover is transient, presumably there is a well-behaved electrical network theory. So we expect that parallel laws should determine all effective resistances uniquely. We have, however, not made this intuition precise. It is not immediately clear whether the uniqueness statement follows from \cite[Theorem 4.9]{aminiFactorials}.
\end{itemize}

\end{Remark}
\begin{proof} 
As before, let $e'$ denote an arbitrary lift of $e$ to the universal cover. Let $\{e'_i \colon 1 \leq i \leq p \}$ be the oriented edges in $T(e')$ whose initial vertex is $e'^+$.The tree $T(e')$ can be decomposed into the union of $e'_i$'s and $T(e'_i)$'s (see Figure~\ref{fig:decompT}).
  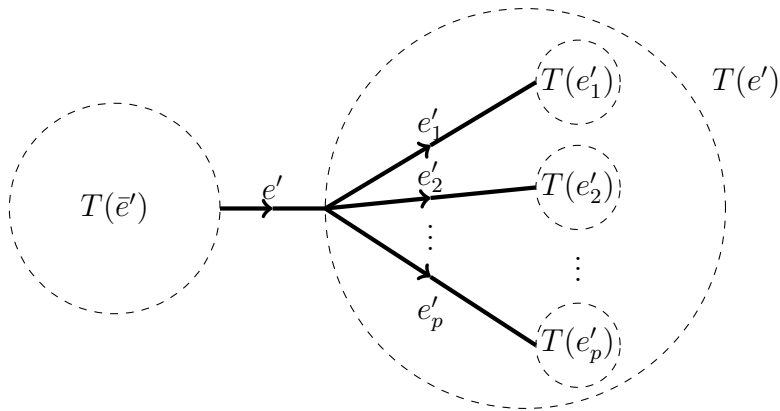
\begin{figure}[H]
   \begin{tikzpicture}[scale=1.4]
     \draw[black, ultra thick, ->](1,0)--(1.5,0);
     \draw[black, ultra thick, -](1.5,0)--(2,0);
     \node at (1.5,0.2){$e'$};
     \draw[dashed] (3.9,0) circle(1.9);
     \draw[black, ultra thick, ->](2,0)--(3,.6);
     \draw[black, ultra thick, -](3,.6)--(4,1.2);
     \draw[black, ultra thick, ->](2,0)--(3,0.1);
     \draw[black, ultra thick, -](3,0.1)--(4,.2);
     \draw[black, ultra thick, ->](2,0)--(3,-.65);
     \draw[black, ultra thick, -](3,-.65)--(4,-1.3);
     \node at (3, 0.8){$e'_1$};
     \node at (3, .3){$e'_2$};
     \node at (3, -1){$e'_p$};
     \node at (3, -.2){$\vdots$};
     \node at (6, 1.2){$T(e')$};
     \draw[dashed] (4.4, 1.2)circle(0.4);
     \draw[dashed] (4.4, 0.2)circle(0.4);
     \draw[dashed] (4.4, -1.3)circle(0.4);
     \node at (4.4, 1.2){$T(e'_1)$};
     \node at (4.4, .2){$T(e'_2)$};
     \node at (4.4, -1.3){$T(e'_p)$};
          \node at (4.4, -.5){$\vdots$};
         \node at (0, 0){$T(\bar{e}')$}; 
           \draw[dashed] (0, 0) circle(1);
\end{tikzpicture}         
\caption{\label{fig:decompT} A decomposition of $T(e')$.}
\end{figure}

Let $\nu_{e'}$ and $\nu_{e'_i}$ be as in Remark~\ref{rmk:Rprop}~(ii), so $R(e) = \|\nu_{e'}\|^2$ and $R(e_i) = \|\nu_{e'_i}\|^2$ for all $e_i \in S_e$.
Let $\omega_i$ be the restriction of $\nu_{e'}$ to the connected component of $T(e') \backslash e'^+$ containing $T(e'_i)$. Let $\beta_i = de'_i-\nu_{e'_i}$. Both $\omega_i$ and $\beta_i$ are harmonic on the connected component of $T(e') \backslash e'^+$ that contains $T(e'_i)$. They both are the unique norm minimizers among harmonic forms with a prescribed valuations on $e'_i$. Therefore we must have $\omega_i = a_i \beta_i$ for some $a_i \in \RR$.
We conclude:
	\[\nu_{e'} = \omega_i = \sum a_i \beta_i = \sum a_i (de'_i-\nu_{e'_i}) \, .\]
	The condition that $d^*\nu_{e'}=\delta_{e'^+}$ implies that $\sum a_i=1$. Finally, we have:
	\[R(e)= \min_{\sum a_i=1}\|\sum a_i (de'_i-\nu_{e'_i})\|^2  
	= \min_{\sum a_i=1}\sum_i a_i^2\left(\ell(e_i)+R(e_i)\right)
	=\left({\sum{1\over \ell(e_i)+R(e_i)}}\right)^{-1} \, .\]
	The last step is by the Cauchy--Schwarz inequality.
\end{proof}

We finish this section with the following amusing example, which shows that the induced ``hyperbolic measure'' on very simple metric graphs can be highly nontrivial.

\begin{Example}\label{ex:amusing}
 Consider the banana graph as in Figure~\ref{fig:hyperbolic.ex} (left). We use Theorem~\ref{cal1} and Theorem~\ref{cal2} to compute the measure induced from the canonical measure on the universal cover. 
\begin{figure}[h!]
$$
\begin{xy}
(0,0)*+{
	\scalebox{.9}{$
	\begin{tikzpicture}
	\draw[black, ultra thick, -] (0,1.2) to [out=-45,in=90] (.6,-.05);
	\draw[black, ultra thick] (.6,0) to [out=-90,in=45] (0,-1.2);
	\draw[black, ultra thick, -] (0,1.2) to [out=-135,in=90] (-.6,-.05);
	\draw[black, ultra thick] (-.6,0) to [out=-90,in=135] (0,-1.2);
	\draw[black, ultra thick, -] (0,1.2) -- (0,0);
	\draw[black, ultra thick] (0,0.1) -- (0,-1.2);
	\fill[black] (0,1.2) circle (.1);
	\fill[black] (0,-1.2) circle (.1);
	\end{tikzpicture}
	$}
};
(-8,0)*+{\mbox{{\smaller $2$}}};
(-2,0)*+{\mbox{{\smaller $1$}}};
(8,0)*+{\mbox{{\smaller $1$}}};
(-6,7)*+{\mbox{{\smaller $e_1$}}};
(2.4,4)*+{\mbox{{\smaller $e_2$}}};
(6,7)*+{\mbox{{\smaller $e_3$}}};

\end{xy}
\ \ \ \ \ \ \ \ \ \ 
\begin{xy}
(0,0)*+{
	\scalebox{.9}{$
	\begin{tikzpicture}
	\draw[black, ultra thick, -] (0,1.2) to [out=-45,in=90] (.6,-.05);
	\draw[black, ultra thick] (.6,0) to [out=-90,in=45] (0,-1.2);
	\draw[black, ultra thick, -] (0,1.2) to [out=-135,in=90] (-.6,-.05);
	\draw[black, ultra thick] (-.6,0) to [out=-90,in=135] (0,-1.2);
	\draw[black, ultra thick, -] (0,1.2) -- (0,0);
	\draw[black, ultra thick] (0,0.1) -- (0,-1.2);
	\fill[black] (0,1.2) circle (.1);
	\fill[black] (0,-1.2) circle (.1);
	\end{tikzpicture}
	$}
};
(-8,0)*+{\mbox{{\smaller $\ell_1$}}};
(2.5,0)*+{\mbox{{\smaller $\ell_2$}}};
(8.8,0)*+{\mbox{{\smaller $\ell_3$}}};
\end{xy}
\ \ \ \ \ \ \ \ 
\!\!\!\!\!\!\!\!\!\!\!\!\!\!\!\!\!\!\!
$$
\caption{\label{fig:hyperbolic.ex}
  Left: A metric graph $\Gamma$.
  Right: the associated ``hyperbolic'' lengths on $\Gamma$ will assign $\ell_1 = {(11-\sqrt{41})}/{10}$ and $\ell_2=\ell_3 = {(\sqrt{41}-1)}/{20}$.}
\end{figure}
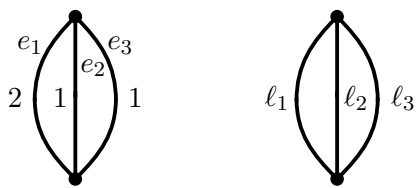
By symmetry, for every oriented edge $e$ we have $R(e) = R(\bar{e})$. Let $x_i=R(e_i) = R(\bar{e}_i)$. By \eqref{conc2} we need to solve the following equations:
\[
{1\over x_1}={1\over 1+x_2}+{1\over 1+x_3} \ , \ 
{1\over x_2}={1\over 2+x_1}+{1\over 1+x_3} \ , \ 
{1\over x_3}={1\over 2+x_1}+{1\over 1+x_2} \, .
\]
With a routine computation, one finds this system has only one positive solution:
\[
x_1={(3+\sqrt{41})}/{8} 
\ , \ 
x_2 = x_3 = {(\sqrt{41}-1)}/{4} \, .
\]
By \eqref{conc1} we compute: 
\[
\mu_{\phi, \can}(e_1) = {(11-\sqrt{41})/ 10}
\ , \
\mu_{\phi, \can}(e_2) = \mu_{\phi, \can}(e_3) = {(\sqrt{41}-1)/ 20} \,.
\]
\end{Example}

\subsubsection{Relation to Poisson--Jensen and equilibrium measures} \label{sec:PJEquil}
There is a relationship between the measure inherited from the canonical measure on the universal cover, and the Poisson-Jensen (and equilibrium) measure. 

Let $\Gamma'$ be the universal cover of compact metric graph $\Gamma$. 
Let $\partial \Gamma'$ denote the {\em ends} (Gromov boundary) of $\Gamma'$, and let $\hat{\Gamma}' = \Gamma' \cup \partial \Gamma'$ denote the {end compactification} of $\Gamma'$.

We want to define, for any $x \in \Gamma'$, a Radon measure $\mu_x$ on $\partial \Gamma' = \hat{\Gamma}' \backslash \Gamma'$. It suffices to describe $\mu_x$ on a base for the topology on $\partial \Gamma'$. This is because any Borel probability measure on a locally compact Hausdorff space with a countable base for its topology is regular. 

Let $G'$ be a model for $\Gamma'$. For $e' \in \EE(G')$, let $T(e')$ be as defined in \S\ref{sec:ucover1}. Let $\partial T(e')$ denote the {\em ends} of $T(e')$. Then $\{\partial T(e') \colon e'\in \EE(G') \}$ gives a base for the topology on $\partial \Gamma'$.

Let $\omega_x \in C^1_{L^2}(\Gamma', \RR)$ be the unique element that minimizes the norm on the affine subspace $\{\alpha \colon d^*\omega=\delta_x \}$. The existence and uniqueness of $\omega_x$ follows from the Hilbert projection theorem.

\begin{Definition} \label{def:pj}
Let $\mu_x$ be the probability measure on $\partial \Gamma'$ defined by 
\[
\mu_x(\partial T(e'))=(\omega_x)_{e'} = \frac{\omega_x(e')}{\ell(e')}
\]
for all $e' \in \EE(G)$. 
\end{Definition}

\begin{Remark}
The measure $\mu_x$ defined above can be interpreted as a {\em Poisson-Jensen measure} and as an {\em equilibrium measure}:

\begin{itemize}
\item[(i)] The measure $\mu_x$ is the {\em Poisson--Jensen measure} for $(\Gamma' , \partial \Gamma')$; it is the unique measure supported on $\partial \Gamma'$ satisfying: 
\begin{itemize}
\item[(a)] $\mu_x$ depends only on $x$, and 
\item[(b)] for any harmonic function $f$ on $\Gamma'$ that extends continuously to $\partial \Gamma'$ and $df$ is $L^2$ integrable, and for any $x \in \Gamma$ we have
\[
f(x) = \int_{\partial \Gamma'} f d\mu_x \, .
\]
\end{itemize}
\item[(ii)]  The measure $\mu_x$ is the {\em equilibrium measure} for $\partial \Gamma'$ relative to $x$; it is the unique measure satisfying: 
\begin{itemize}
\item[(a)] $\mu_x$ is a Borel probability measure, and 
\item[(b)] among all $\nu$ Borel probability measures on $\partial{\Gamma}'$, the measure $\mu_x$ it is the unique measure minimizing the {\em energy integral}:
\[
I_x(\nu) = \iint_{\partial{\Gamma}' \times \partial{\Gamma}'} {\rm dist}_x( y, z) d\nu(y)  d\nu(z) \, ,
\]
where ${\rm dist}_x(y, z) = (y, z)_x$ is the usual {\em Gromov product} on the metric space $\hat{\Gamma}'$, defined as the length of the intersection of the geodesic rays from $x$ to $y$ and from $x$ to $z$.
\end{itemize}
\end{itemize}
We do not use these interpretations, so we skip the (straightforward) proofs. See also \cite{aminiFactorials} for an interesting connection to generalized factorial sequences.
\end{Remark}

Our final result relates the canonical measure on universal covers to the measure $\mu_x$ of Definition~\ref{def:pj}. For a signed measure $\mu$, the total variation of $\mu$ is denoted by $|\mu|$. 
\begin{Proposition}\label{PJ} 
	Let $\Gamma'$ be the universal cover of a compact metric graph $\Gamma$ with $g(\Gamma) \geq 2$. For $x \in \Gamma'$, let $\mu_x$ be the measure on $\partial \Gamma'$ as in Definition~\ref{def:pj}. Then
	\begin{itemize}
	\item[(a)] Let $e$ be an edge in $\Gamma$ and $e'$ be a lift of $e$ to $\Gamma'$. Then
	\[
	\mu_{\phi, \can} (e) = \frac{1}{2} |\mu_{e'^+} - \mu_{e'^-} | (\partial \Gamma') \, .
	\]
	\item[(b)] For any line segment $e' \in \Gamma'$ and $x \in e'$ we have 
	\[\mu_{can}^{\Gamma'}|_{e'}  = \left({1\over 2}\left|{d\over dx}\mu_x\right|(\partial \Gamma')\right)dx \, .\]
 Here ${d\over dx}\mu_x$ denotes the the signed measure $\lim_{\delta\rightarrow 0}{1\over\delta}(\mu_{x+\delta \mathbf{v}}-\mu_x)$, where $\mathbf{v}$ is any unit tangent direction at $x$.
 \end{itemize} 
\end{Proposition}
\begin{proof}
  Let $G'$ be any model for $\Gamma'$ compatible with the covering map. For $e' \in \EE(G')$ we let $T(e')$ and $\nu_{e'}$ be as in \S\ref{sec:ucover1}. Since $\omega_{e'^-}$ has minimum norm, it must be proportional to $\nu_{e'}$ when restricted to $T({e'})$ and proportional to $d\overline{e}'-\nu_{\overline{e}'}$ when restricted to $T(\overline{e}') \cup e'$. So
  \[\omega_{e'^+}=a \nu_{e'}+(1-a) (d\overline{e}'-\nu_{\overline{e}'})  \]
for some real number $a$. Similarly,
  \[\omega_{e'^-}=b \nu_{\overline{e}'} +(1-b)(de'-\nu_{e'})\]
  for some real number $b$. The values of $a$ and $b$ must minimize the quantities  
  \[
  \|\omega_{e'^+}\|^2 = a^2R(e)+(1-a)^2(R(\overline{e})+\ell(e))
  \quad \text{and} \quad
      \|\omega_{e'^-}\|^2 = b^2R(e)+(1-b)^2(R({e})+\ell(e)) \,,
  \]
  so one computes:
  \[a={R(\overline{e})+\ell(e)\over R(e)+R(\overline{e})+\ell(e)}
  \quad \text{and} \quad
 b={R(e)+\ell(e)\over R(e)+R(\overline{e})+\ell(e)} \,.
 \]
  Therefore,
    \[
  \begin{aligned}
  |\mu_{e'^+}-\mu_{e'^-}|(\partial\Gamma') &=|\mu_{e'^+}(\partial T(e'))-\mu_{e'^-}(\partial T(e'))|+|\mu_{\overline{e}'^+}(\partial T(\overline{e}'))-\mu_{\overline{e}'^-}(\partial T(\overline{e}'))|\\
&=((\omega_{e'^+})_{e'} + (\omega_{e'^-})_{e'}) + 
    ((\omega_{e'^-})_{\overline{e}'} + (\omega_{e'^+})_{\overline{e}'}) \\
&=2(1-(\omega_{e'^+})_{e'}+(\omega_{e'^-})_{e'}) \\
&=2(1-(1-a)-(1-b))\\
&= 2 {\ell(e)}/{(\mathscr{S}(e)+\ell(e))}\\
&=2 \mu_{\phi, \can}(e) \,.
\end{aligned}
\]
The last equality is by \eqref{conc1}. 
Part (b) follows from part (a) after refining $G'$. 
\end{proof}


\begin{bibdiv}
\begin{biblist}

\bib{ABBR1}{article}{
      author={Amini, Omid},
      author={Baker, Matthew},
      author={Brugall\'e, Erwan},
      author={Rabinoff, Joseph},
       title={Lifting harmonic morphisms {I}: metrized complexes and
  {B}erkovich skeleta},
        date={2015},
        ISSN={2197-9847},
     journal={Res. Math. Sci.},
      volume={2},
       pages={Art. 7, 67},
  url={https://doi-org.proxy.library.cornell.edu/10.1186/s40687-014-0019-0},
      review={\MR{3375652}},
}

\bib{ABBR2}{article}{
      author={Amini, Omid},
      author={Baker, Matthew},
      author={Brugall\'e, Erwan},
      author={Rabinoff, Joseph},
       title={Lifting harmonic morphisms {II}: {T}ropical curves and metrized
  complexes},
        date={2015},
        ISSN={1937-0652},
     journal={Algebra Number Theory},
      volume={9},
      number={2},
       pages={267\ndash 315},
  url={https://doi-org.proxy.library.cornell.edu/10.2140/ant.2015.9.267},
      review={\MR{3320845}},
}

\bib{ABBGNRS}{article}{
      author={Abert, Miklos},
      author={Bergeron, Nicolas},
      author={Biringer, Ian},
      author={Gelander, Tsachik},
      author={Nikolov, Nikolay},
      author={Raimbault, Jean},
      author={Samet, Iddo},
       title={On the growth of {$L^2$}-invariants for sequences of lattices in
  {L}ie groups},
        date={2017},
        ISSN={0003-486X},
     journal={Ann. of Math. (2)},
      volume={185},
      number={3},
       pages={711\ndash 790},
  url={https://doi-org.proxy.library.cornell.edu/10.4007/annals.2017.185.3.1},
      review={\MR{3664810}},
}

\bib{ABKS}{article}{
      author={An, Yang},
      author={Baker, Matthew},
      author={Kuperberg, Greg},
      author={Shokrieh, Farbod},
       title={Canonical representatives for divisor classes on tropical curves
  and the matrix-tree theorem},
        date={2014},
        ISSN={2050-5094},
     journal={Forum Math. Sigma},
      volume={2},
       pages={e24, 25},
         url={https://doi-org.proxy.library.cornell.edu/10.1017/fms.2014.25},
      review={\MR{3264262}},
}

\bib{ACFK}{article}{
      author={Ab\'ert, Mikl\'os},
      author={Csikv\'ari, P\'eter},
      author={Frenkel, P\'eter~E.},
      author={Kun, G\'abor},
       title={Matchings in {B}enjamini-{S}chramm convergent graph sequences},
        date={2016},
        ISSN={0002-9947},
     journal={Trans. Amer. Math. Soc.},
      volume={368},
      number={6},
       pages={4197\ndash 4218},
         url={https://doi-org.proxy.library.cornell.edu/10.1090/tran/6464},
      review={\MR{3453369}},
}

\bib{amini2014equidistribution}{unpublished}{
      author={Amini, Omid},
       title={Equidistribution of {W}eierstrass points on curves over
  non-{A}rchimedean fields},
        date={2014},
         url={http://arxiv.org/abs/1412.0926},
        note={Preprint available at \href{https://arxiv.org/abs/1412.0926}{{\tt
  ar{X}iv:1412.0926}}},
}

\bib{aminiFactorials}{unpublished}{
      author={Amini, Omid},
       title={Logarithmic tree factorials},
        date={2016},
         url={http://arxiv.org/abs/1611.02142},
        note={Preprint available at
  \href{https://arxiv.org/abs/1611.02142}{{\tt ar{X}iv:1611.02142}}},
}

\bib{Andre}{book}{
      author={Andr\'e, Yves},
       title={Period mappings and differential equations. {F}rom {$\Bbb C$} to
  {$\Bbb C_p$}},
      series={MSJ Memoirs},
   publisher={Mathematical Society of Japan, Tokyo},
        date={2003},
      volume={12},
        ISBN={4-931469-22-1},
        note={T\^ohoku-Hokkaid\^o lectures in arithmetic geometry, With
  appendices by F. Kato and N. Tsuzuki},
      review={\MR{1978691}},
}

\bib{Berk}{book}{
      author={Berkovich, Vladimir~G.},
       title={Spectral theory and analytic geometry over non-{A}rchimedean
  fields},
      series={Mathematical Surveys and Monographs},
   publisher={American Mathematical Society, Providence, RI},
        date={1990},
      volume={33},
        ISBN={0-8218-1534-2},
      review={\MR{1070709}},
}

\bib{baker2011metric}{article}{
      author={Baker, Matthew},
      author={Faber, Xander},
       title={Metric properties of the tropical {A}bel-{J}acobi map},
        date={2011},
        ISSN={0925-9899},
     journal={J. Algebraic Combin.},
      volume={33},
      number={3},
       pages={349\ndash 381},
         url={http://dx.doi.org/10.1007/s10801-010-0247-3},
      review={\MR{2772537}},
}

\bib{bestvina2014hyperbolicity}{article}{
      author={Bestvina, Mladen},
      author={Feighn, Mark},
       title={Hyperbolicity of the complex of free factors},
        date={2014},
        ISSN={0001-8708},
     journal={Adv. Math.},
      volume={256},
       pages={104\ndash 155},
  url={https://doi-org.proxy.library.cornell.edu/10.1016/j.aim.2014.02.001},
      review={\MR{3177291}},
}

\bib{Biggs}{article}{
      author={Biggs, Norman},
       title={Algebraic potential theory on graphs},
        date={1997},
        ISSN={0024-6093},
     journal={Bull. London Math. Soc.},
      volume={29},
      number={6},
       pages={641\ndash 682},
         url={http://dx.doi.org/10.1112/S0024609397003305},
      review={\MR{1468054}},
}

\bib{BJ}{article}{
      author={Boucksom, S\'ebastien},
      author={Jonsson, Mattias},
       title={Tropical and non-{A}rchimedean limits of degenerating families of
  volume forms},
        date={2017},
        ISSN={2429-7100},
     journal={J. \'Ec. polytech. Math.},
      volume={4},
       pages={87\ndash 139},
      review={\MR{3611100}},
}

\bib{BN07}{article}{
      author={Baker, Matthew},
      author={Norine, Serguei},
       title={Riemann-{R}och and {A}bel-{J}acobi theory on a finite graph},
        date={2007},
        ISSN={0001-8708},
     journal={Adv. Math.},
      volume={215},
      number={2},
       pages={766\ndash 788},
  url={https://doi-org.proxy.library.cornell.edu/10.1016/j.aim.2007.04.012},
      review={\MR{2355607}},
}

\bib{br}{article}{
      author={Baker, Matthew},
      author={Rumely, Robert},
       title={Harmonic analysis on metrized graphs},
        date={2007},
        ISSN={0008-414X},
     journal={Canad. J. Math.},
      volume={59},
      number={2},
       pages={225\ndash 275},
  url={https://doi-org.proxy.library.cornell.edu/10.4153/CJM-2007-010-2},
      review={\MR{2310616}},
}

\bib{baker2007potential}{book}{
      author={Baker, Matthew},
      author={Rumely, Robert},
       title={Potential theory and dynamics on the {B}erkovich projective
  line},
      series={Mathematical Surveys and Monographs},
   publisher={American Mathematical Society, Providence, RI},
        date={2010},
      volume={159},
        ISBN={978-0-8218-4924-8},
         url={http://dx.doi.org/10.1090/surv/159},
      review={\MR{2599526}},
}

\bib{BenSch}{article}{
      author={Benjamini, Itai},
      author={Schramm, Oded},
       title={Recurrence of distributional limits of finite planar graphs},
        date={2001},
        ISSN={1083-6489},
     journal={Electron. J. Probab.},
      volume={6},
       pages={no. 23, 13},
         url={https://doi-org.proxy.library.cornell.edu/10.1214/EJP.v6-96},
      review={\MR{1873300}},
}

\bib{BS13}{article}{
      author={Baker, Matthew},
      author={Shokrieh, Farbod},
       title={Chip-firing games, potential theory on graphs, and spanning
  trees},
        date={2013},
        ISSN={0097-3165},
     journal={J. Combin. Theory Ser. A},
      volume={120},
      number={1},
       pages={164\ndash 182},
  url={https://doi-org.proxy.library.cornell.edu/10.1016/j.jcta.2012.07.011},
      review={\MR{2971705}},
}

\bib{BSW}{unpublished}{
      author={Baik, Hyungryul},
      author={Shokrieh, Farbod},
      author={Wu, Chenxi},
       title={Limits of canonical forms on towers of {R}iemann surfaces},
        date={2018},
         url={http://arxiv.org/abs/1808.00477},
        note={Preprint available at
  \href{https://arxiv.org/abs/1808.00477}{{\tt ar{X}iv:1808.00477}}},
}

\bib{ChLoir}{incollection}{
      author={Chambert-Loir, Antoine},
       title={Heights and measures on analytic spaces. {A} survey of recent
  results, and some remarks},
        date={2011},
   booktitle={Motivic integration and its interactions with model theory and
  non-{A}rchimedean geometry. {V}olume {II}},
      series={London Math. Soc. Lecture Note Ser.},
      volume={384},
   publisher={Cambridge Univ. Press, Cambridge},
       pages={1\ndash 50},
      review={\MR{2885340}},
}

\bib{chinburg1993capacity}{article}{
      author={Chinburg, Ted},
      author={Rumely, Robert},
       title={The capacity pairing},
        date={1993},
        ISSN={0075-4102},
     journal={J. reine angew. Math.},
      volume={434},
       pages={1\ndash 44},
         url={http://dx.doi.org/10.1515/crll.1993.434.1},
      review={\MR{1195689}},
}

\bib{culler1986moduli}{article}{
      author={Culler, Marc},
      author={Vogtmann, Karen},
       title={Moduli of graphs and automorphisms of free groups},
        date={1986},
        ISSN={0020-9910},
     journal={Invent. Math.},
      volume={84},
      number={1},
       pages={91\ndash 119},
         url={https://doi-org.proxy.library.cornell.edu/10.1007/BF01388734},
      review={\MR{830040}},
}



\bib{dowdall2015dynamics}{article}{
      author={Dowdall, Spencer},
      author={Kapovich, Ilya},
      author={Leininger, Christopher~J.},
       title={Dynamics on free-by-cyclic groups},
        date={2015},
        ISSN={1465-3060},
     journal={Geom. Topol.},
      volume={19},
      number={5},
       pages={2801\ndash 2899},
  url={https://doi-org.proxy.library.cornell.edu/10.2140/gt.2015.19.2801},
      review={\MR{3416115}},
}

\bib{DoyleSnell}{book}{
      author={Doyle, Peter~G.},
      author={Snell, J.~Laurie},
       title={Random walks and electric networks},
      series={Carus Mathematical Monographs},
   publisher={Mathematical Association of America, Washington, DC},
        date={1984},
      volume={22},
        ISBN={0-88385-024-9},
      review={\MR{920811}},
}

\bib{Flanders}{article}{
      author={Flanders, Harley},
       title={A new proof of {R}. {F}oster's averaging formula in networks},
        date={1974},
     journal={Linear Algebra and Appl.},
      volume={8},
       pages={35\ndash 37},
      review={\MR{0329772}},
}

\bib{Foster}{incollection}{
      author={Foster, Ronald~M.},
       title={The average impedance of an electrical network},
        date={1948},
   booktitle={Reissner {A}nniversary {V}olume, {C}ontributions to {A}pplied
  {M}echanics},
   publisher={J. W. Edwards, Ann Arbor, Michigan},
       pages={333\ndash 340},
      review={\MR{0029773}},
}

\bib{Feng}{article}{
      author={Gu, Xianfeng},
      author={Guo, Ren},
      author={Luo, Feng},
      author={Sun, Jian},
      author={Wu, Tianqi},
       title={A discrete uniformization theorem for polyhedral surfaces {II}},
        date={2018},
        ISSN={0022-040X},
     journal={J. Differential Geom.},
      volume={109},
      number={3},
       pages={431\ndash 466},
  url={https://doi-org.proxy.library.cornell.edu/10.4310/jdg/1531188190},
      review={\MR{3825607}},
}

\bib{GvdP}{book}{
      author={Gerritzen, Lothar},
      author={van~der Put, Marius},
       title={Schottky groups and {M}umford curves},
      series={Lecture Notes in Mathematics},
   publisher={Springer, Berlin},
        date={1980},
      volume={817},
        ISBN={3-540-10229-9},
      review={\MR{590243}},
}

\bib{Heinz}{article}{
      author={Heinz, Niels},
       title={Admissible metrics for line bundles on curves and abelian
  varieties over non-{A}rchimedean local fields},
        date={2004},
        ISSN={0003-889X},
     journal={Arch. Math. (Basel)},
      volume={82},
      number={2},
       pages={128\ndash 139},
  url={https://doi-org.proxy.library.cornell.edu/10.1007/s00013-003-4744-7},
      review={\MR{2047666}},
}

\bib{handel2013free}{article}{
      author={Handel, Michael},
      author={Mosher, Lee},
       title={The free splitting complex of a free group, {I}: hyperbolicity},
        date={2013},
        ISSN={1465-3060},
     journal={Geom. Topol.},
      volume={17},
      number={3},
       pages={1581\ndash 1672},
  url={https://doi-org.proxy.library.cornell.edu/10.2140/gt.2013.17.1581},
      review={\MR{3073931}},
}

\bib{deJong}{unpublished}{
      author={{d}e Jong, Robin},
       title={{F}altings delta-invariant and semistable degeneration},
        date={2015},
         url={https://arxiv.org/abs/1511.06567},
        note={To appear in {\em J. Differential Geom.} Preprint available at
  \href{https://arxiv.org/abs/1511.06567}{{\tt ar{X}iv:1511.06567}}},
}

\bib{Kazhdan}{incollection}{
      author={Kazhdan, David},
       title={Arithmetic varieties and their fields of quasi-definition},
        date={1971},
   booktitle={Actes du {C}ongr\`es {I}nternational des {M}ath\'ematiciens
  ({N}ice, 1970), {T}ome 2},
   publisher={Gauthier-Villars, Paris},
       pages={321\ndash 325},
      review={\MR{0435081}},
}

\bib{Kaj}{inproceedings}{
      author={Kazhdan, David},
       title={On arithmetic varieties},
        date={1975},
   booktitle={Lie groups and their representations ({P}roc. {S}ummer {S}chool,
  {B}olyai {J}\'anos {M}ath. {S}oc., {B}udapest, 1971)},
   publisher={Halsted, New York},
       pages={151\ndash 217},
      review={\MR{0486316}},
}

\bib{Kirchhoff}{article}{
      author={Kirchhoff, Gustav},
       title={Ueber die {A}ufl{\"o}sung der {G}leichungen, auf welche man bei
  der {U}ntersuchung der linearen {V}ertheilung galvanischer {S}tr{\"o}me
  gef{\"u}hrt wird},
        date={1847},
     journal={Annalen der Physik},
      volume={148},
      number={12},
       pages={497\ndash 508},
}

\bib{luck}{book}{
      author={L\"uck, Wolfgang},
       title={{$L^2$}-invariants: theory and applications to geometry and
  {$K$}-theory},
      series={Ergebnisse der Mathematik und ihrer Grenzgebiete. 3. Folge. A
  Series of Modern Surveys in Mathematics},
   publisher={Springer-Verlag, Berlin},
        date={2002},
      volume={44},
        ISBN={3-540-43566-2},
         url={http://dx.doi.org/10.1007/978-3-662-04687-6},
      review={\MR{1926649}},
}

\bib{Lepage}{article}{
      author={Lepage, Emmanuel},
       title={Tempered fundamental group and metric graph of a {M}umford
  curve},
        date={2010},
        ISSN={0034-5318},
     journal={Publ. Res. Inst. Math. Sci.},
      volume={46},
      number={4},
       pages={849\ndash 897},
      review={\MR{2791009}},
}

\bib{LP:book}{book}{
      author={Lyons, Russell},
      author={Peres, Yuval},
       title={Probability on trees and networks},
      series={Cambridge Series in Statistical and Probabilistic Mathematics},
   publisher={Cambridge University Press, New York},
        date={2016},
      volume={42},
        ISBN={978-1-107-16015-6},
         url={http://dx.doi.org/10.1017/9781316672815},
      review={\MR{3616205}},
}

\bib{LyonsR}{article}{
      author={Lyons, Russell},
       title={Determinantal probability measures},
        date={2003},
        ISSN={0073-8301},
     journal={Publ. Math. Inst. Hautes \'Etudes Sci.},
      number={98},
       pages={167\ndash 212},
  url={https://doi-org.proxy.library.cornell.edu/10.1007/s10240-003-0016-0},
      review={\MR{2031202}},
}

\bib{Lyons}{article}{
      author={Lyons, Terry},
       title={A simple criterion for transience of a reversible {M}arkov
  chain},
        date={1983},
        ISSN={0091-1798},
     journal={Ann. Probab.},
      volume={11},
      number={2},
       pages={393\ndash 402},
  url={http://links.jstor.org/sici?sici=0091-1798(198305)11:2<393:ASCFTO>2.0.CO;2-S&origin=MSN},
      review={\MR{690136}},
}

\bib{magnus}{article}{
      author={Magnus, W.},
       title={Residually finite groups},
        date={1969},
        ISSN={0002-9904},
     journal={Bull. Amer. Math. Soc.},
      volume={75},
       pages={305\ndash 316},
         url={http://dx.doi.org/10.1090/S0002-9904-1969-12149-X},
      review={\MR{0241525}},
}

\bib{mann2014hyperbolicity}{article}{
      author={Mann, Brian},
       title={Hyperbolicity of the cyclic splitting graph},
        date={2014},
        ISSN={0046-5755},
     journal={Geom. Dedicata},
      volume={173},
       pages={271\ndash 280},
  url={https://doi-org.proxy.library.cornell.edu/10.1007/s10711-013-9941-3},
      review={\MR{3275303}},
}

\bib{mcmullen2013entropy}{article}{
      author={McMullen, Curtis~T.},
       title={Entropy on {R}iemann surfaces and the {J}acobians of finite
  covers},
        date={2013},
        ISSN={0010-2571},
     journal={Comment. Math. Helv.},
      volume={88},
      number={4},
       pages={953\ndash 964},
         url={http://dx.doi.org/10.4171/CMH/308},
      review={\MR{3134416}},
}

\bib{McMullenEntr}{article}{
      author={McMullen, Curtis~T.},
       title={Entropy and the clique polynomial},
        date={2015},
        ISSN={1753-8416},
     journal={J. Topol.},
      volume={8},
      number={1},
       pages={184\ndash 212},
         url={https://doi-org.proxy.library.cornell.edu/10.1112/jtopol/jtu022},
      review={\MR{3335252}},
}

\bib{mumford}{book}{
      author={Mumford, David},
       title={Curves and their {J}acobians},
   publisher={The University of Michigan Press, Ann Arbor, Mich.},
        date={1975},
      review={\MR{0419430}},
}

\bib{MK}{incollection}{
      author={Mikhalkin, Grigory},
      author={Zharkov, Ilia},
       title={Tropical curves, their {J}acobians and theta functions},
        date={2008},
   booktitle={Curves and abelian varieties},
      series={Contemp. Math.},
      volume={465},
   publisher={Amer. Math. Soc., Providence, RI},
       pages={203\ndash 230},
  url={https://doi-org.proxy.library.cornell.edu/10.1090/conm/465/09104},
      review={\MR{2457739}},
}

\bib{Neeman}{article}{
      author={Neeman, Amnon},
       title={The distribution of {W}eierstrass points on a compact {R}iemann
  surface},
        date={1984},
        ISSN={0003-486X},
     journal={Ann. of Math. (2)},
      volume={120},
      number={2},
       pages={317\ndash 328},
         url={https://doi-org.proxy.library.cornell.edu/10.2307/2006944},
      review={\MR{763909}},
}

\bib{pansu1993introduction}{incollection}{
      author={Pansu, Pierre},
       title={Introduction to {$L^2$} {B}etti numbers},
        date={1996},
   booktitle={Riemannian geometry ({W}aterloo, {ON}, 1993)},
      series={Fields Inst. Monogr.},
      volume={4},
   publisher={Amer. Math. Soc., Providence, RI},
       pages={53\ndash 86},
      review={\MR{1377309}},
}

\bib{rhodes1993sequences}{article}{
      author={Rhodes, John~A.},
       title={Sequences of metrics on compact {R}iemann surfaces},
        date={1993},
        ISSN={0012-7094},
     journal={Duke Math. J.},
      volume={72},
      number={3},
       pages={725\ndash 738},
         url={http://dx.doi.org/10.1215/S0012-7094-93-07227-4},
      review={\MR{1253622}},
}

\bib{Sho10}{article}{
      author={Shokrieh, Farbod},
       title={The monodromy pairing and discrete logarithm on the {J}acobian of
  finite graphs},
        date={2010},
        ISSN={1862-2976},
     journal={J. Math. Cryptol.},
      volume={4},
      number={1},
       pages={43\ndash 56},
         url={https://doi-org.proxy.library.cornell.edu/10.1515/JMC.2010.002},
      review={\MR{2660333}},
}

\bib{zhang1993admissible}{article}{
      author={Zhang, Shouwu},
       title={Admissible pairing on a curve},
        date={1993},
        ISSN={0020-9910},
     journal={Invent. Math.},
      volume={112},
      number={1},
       pages={171\ndash 193},
         url={http://dx.doi.org/10.1007/BF01232429},
      review={\MR{1207481}},
}

\end{biblist}
\end{bibdiv}
\end{document}